\documentclass[runningheads,11pt]{llncs}

\usepackage{amssymb}
\usepackage{graphicx}
\usepackage{mathrsfs}
\usepackage{hyperref}
\usepackage{color}
\usepackage{amsmath}
\usepackage{url}
\usepackage{cite}

\urldef{\mailsa}\path|liliping@amss.ac.cn|    
\urldef{\mailsb}\path|loushuwen@gmail.com|
\newcommand{\keywords}[1]{\par\addvspace\baselineskip
\noindent\keywordname\enspace\ignorespaces#1}

\def\sE{{\mathscr E}}
\def\sF{{\mathscr F}}
\def\sC{{\mathscr C}}
\def\cB{{\mathcal{B}}}
\def\cE{{\mathcal{E}}}
\def\cF{{\mathcal{F}}}
\def\cL{{\mathcal{L}}}
\def\cA{{\mathcal{A}}}
\def\bR{{\mathbb{R}}}
\def\sG{{\mathscr G}}
\def\sA{{\mathscr A}}
\def\sL{{\mathscr L}}
\def\cD{{\mathcal{D}}}
\def\fR{{\mathfrak{R}}}
\def\bP{{\mathbb{P}}}
\def\tp{{\mathtt{p}}}

\def\bfP{{\mathbf{P}}}
\def\bfQ{{\mathbf{Q}}}

\def\rc{{\mathrm{c}}}
\def\bN{{\mathbb{N}}}

\def\ts{{\mathtt{s}}}

\def\re{{\mathrm{e}}}
\def\fm{{\mathfrak{m}}}

\def\b0{{\textbf{0}}}

\def\eps{\varepsilon}
\def\wh{\widehat}

\def\<{\langle}
\def\>{\rangle}

\newcommand{\IE}{{\mathbb{E}}}
\newcommand{\IP}{{\mathbb{P}}}
\newcommand{\IR}{{\mathbb{R}}}

\begin{document}


\title{Distorted Brownian motions on space with varying dimension}

\titlerunning{Distorted Brownian motions with varying dimension}

%
%

\author{Liping Li}
\institute{RCSDS, HCMS, Academy of Mathematics and Systems Science, Chinese Academy of Sciences, Beijing 100190, China.}

\author{Liping Li%
\thanks{The first named author is partially supported by NSFC (No. 11688101, No. 11801546 and No. 11931004) and Key Laboratory of Random Complex Structures and Data Science, Academy of Mathematics and Systems Science, Chinese Academy of Sciences (No. 2008DP173182).}%
\and Shuwen Lou
}
\authorrunning{L. Li, S. Lou}

\institute{RCSDS, HCMS, Academy of Mathematics and Systems Science,\\ Chinese Academy of Sciences, Beijing, China.\\
\mailsa\\
$\quad$\\
Loyola University Chicago, Chicago, USA \\
\mailsb\\
}

%
%

\maketitle

\begin{abstract}
Roughly speaking, a space with varying dimension consists of at least two components with different dimensions. In this paper we will concentrate on the one, which can be treated as $\bR^3$ tying a half line not contained by $\bR^3$ at the origin. The aim is twofold. On one hand, we will introduce so-called distorted Brownian motions on this space with varying dimension (dBMVDs in abbreviation) and study their basic properties by means of Dirichlet forms. On the other hand, we will prove the joint continuity of the transition density functions of these dBMVDs and derive the short-time heat kernel estimates for them.

\end{abstract}

\section{Introduction}\label{Sec-intro}

The concept of \emph{distorted Brownian motion} (dBM in abbreviation) arises in mathematical physics. As in e.g. \cite{AHS77}, it is introduced on a usual Euclidean space $\bR^d$ by an energy form given by the closure of the quadratic form on $L^2(\bR^d,m)$:
\begin{equation}\label{eq-def-dbm}
	\mathcal{E}(f,g):=\frac{1}{2}\int_{\bR^d}\nabla f(x)\nabla g(x)m(dx), \quad f,g\in C_c^\infty(\bR^d),
\end{equation}
where $m(dx)=\phi(x)^2dx$ is a certain Radon measure on $\bR^d$. This energy form is indeed a so-called regular Dirichelt form, and according to a well-known one-to-one correspondence by Fukushima (see e.g. \cite{FOT}), the associated diffusion process of it is named the distorted Brownian motion on $\bR^d$. The study of this 
stochastic model has been applied to many different physical fields. For example, it was utilized to analyse point interactions in quantum mechanics as in \cite{AGH05}, and Cranston et al. (see e.g. \cite{CKMV1, CKMV2, CM}) brought it to certain polymer models in statistical physics to observe related phenomena of phase transitions. This diffusion process also receives considerable attention of mathematical researchers. It plays an important role in the theory of stochastic analysis, stochastic differential equations (SDEs in abbreviation), and etc. The generalization of dBM to the one on infinite dimensional state space is also of great value to the study of infinite dimensional stochastic analysis, see e.g. \cite{RT15} and the references therein. 

What is a space with varying dimension? Roughly speaking, it consists of at least two components with different dimensions tying together. For example, $\bR^d$ tying a half line not contained by $\bR^d$ at the origin is exactly a space with varying dimension $(1,d)$ for $d\geq 2$. More mathematically,
\begin{equation}\label{eq-E-def}
	E:=\left\{(x_1,\cdots, x_d, r)\in \bR^{d+1}: x_1=\cdots =x_d=0, r\geq 0\right \} \cup \left\{(x_1,\cdots, x_d, r)\in \bR^{d+1}: r=0\right\}
\end{equation} 
is such an example. It was the second named author and her co-author in \cite{CL}, who initiated the study of stochastic model of Brownian motions on the space with varying dimension $E$ (BMVD in abbreviation) for $d=2$. The crucial challenge thereof is that a two-dimensional Brownian motion does not hit the origin, resulting in that the desired model starting from the two-dimensional part of $E$ can hardly reach the tying half line. To overcome it, a so-called ``darning" method was utilized in \cite{CL} to set the resistance on a two-dimensional disc centred at the origin equal to zero. At a heuristic level, this method collapses this disc into an abstract point, which the half line (becoming a pole indeed) is tied to rather than the origin. Since a two-dimensional Brownian motion can hit every disc, the ``darning" BMVD has a chance to climb the tying pole. 

In this paper, we aim to introduce and study \emph{distorted Brownian motions on the state space with varying dimension} $E$ (dBMVDs in abbreviation). To do this, it is still a challenge that a distorted Brownian motion on $\bR^d$ for $d\geq 2$ seldom hits a singleton. Analogical darning method definitely works but is somewhat cumbersome, since it may make the consideration much involved as we see in \cite{CL, Lou}. Fortunately, the presence of flexible density function $\phi^2$ in \eqref{eq-def-dbm} makes us in the hope of finding out suitable multi-dimensional models, which can hit the origin at times, as the candidates of multi-dimensional building block of dBMVD. Such examples given by the density functions
\begin{equation}\label{eq-def-density}
\phi(x)=\psi_\gamma(x):=\frac{\re^{-\gamma |x|}}{2\pi |x|},\quad x\in \bR^3
\end{equation}
with $d=3$ and a positive constant $\gamma>0$ appeared in a series of papers of Cranston et. al (see e.g. \cite{CM, CKMV1, CKMV2}) to model the probabilistic counterparts of phase transition phenomena in certain polymer models. Later in \cite{FL17}, Fitzsimmons and the first named author described the associated process $X^3$ (the superscript ``$3$'' stands for dimension) in terms of Dirichlet forms. Heuristically, $X^3$ corresponds to an informal Schr\"odinger operator
\begin{equation}\label{eq-def-cL}
	\mathcal{L}_\gamma =\frac{1}{2}\Delta +\beta_\gamma \cdot \delta_0,
\end{equation}
where $\Delta$ is the three-dimensional Laplacian, $\beta_\gamma$ is a certain constant depending on $\gamma$ and $\delta_0$ is the Dirac function at the origin. Rigorous interpretation of $\cL_\gamma$ will be stated in \S\ref{SEC22}. Because of the infinite potential manifested by $\delta_0$, the process $X^3$ feels a strong push towards the origin. As a consequence, although $X^3$ obviously does not hit the singleton $\{x\}$ for $x\neq 0$, it turns out in \cite[Proposition~3.1]{FL17} that $X^3$ starting from everywhere can reach the origin in finite time with positive probability (in fact, with probability $1$!). With these three-dimensional processes at hand, we are motivated to consider stochastic models on $E$ with $d=3$ that arise more naturally. It is worth pointing out that for $d=1$, under a mild condition that $1/\phi^2$ is locally integrable, a one-dimensional distorted Brownian motion $X^1$ is always irreducible in the sense that for all $x,y\in \bR$, $X^1$ starting from $x$ can reach $y$ in finite time (see e.g. \cite{L18, LY172}). As a result, we have a host of processes as the candidates of one-dimensional building block of dBMVD.

Below we give a general description to the dBMVDs that interest us in this paper. Let $E$ be in \eqref{eq-E-def} with $d=3$, and denote its one-dimensional part and three-dimensional part by $\fR_+$ and $\fR^3$ respectively. Clearly $\fR_+\cong \bR_+:=[0,\infty)$ and $\fR^3\cong \bR^3$. The isomorphic mapping from $\bR_+$ (resp. $\bR^3$) to $\fR_+$ (resp. $\fR^3$) is denoted by $\iota_+$ (resp. $\iota_3$). A first step towards dBMVD on $E$ is to introduce a one-dimensional dBM $X^+$ on $\bR_+$ and a three-dimensional dBM $X^3$ on $\bR^3$ respectively as follows: $X^+$ is given by a regular Dirichlet form $(\cE^+,\cF^+)$ on $L^2(\bR_+,m_+):=L^2(\bR_+, \rho(r)dr)$:
\[
\begin{aligned}
	&\cF^+:=\left\{f\in L^2(\bR_+,m_+): f'\in L^2(\bR_+,m_+) \right\}, \\
	&\cE^+(f,g):=\frac{1}{2}\int_{\bR_+}f'(r)g'(r)m_+(dr),\quad f,g\in \cF^+,
\end{aligned}
\]
where $\rho$ is a positive, locally integrable function such that $1/\rho\in L^1_\mathrm{loc}(\bR_+)$. This diffusion is irreducible and reflecting at $0\in \bR_+$. As mentioned above, $X^3$ is determined by the energy form \eqref{eq-def-dbm} with $d=3$ and $\phi$ in \eqref{eq-def-density}, whereas $\gamma$ is taken to be an arbitrary real number. We should point out that almost all statements about $X^3$ in \cite{FL17} can be extended to the case $\gamma\leq 0$. For readers' convenience, some results concerning $X^3$ for all $\gamma\in \bR$ will be repeated in \S\ref{SEC2-1}. Particularly, the origin is of positive capacity relative to $X^3$, so that $X^3$ can always reach the origin in finite time. A new theorem, i.e. Theorem~\ref{THM3}, characterizes the relation between $X^3$ and three-dimensional Brownian motion (via $h$-transform), as will be utilized to estimate the heat kernel of dBMVD in \S\ref{Sec-short-time-HKE}. By embedding these two dBMs in $E$, we obtain $M^+:=\iota_+(X^+)$ and $M^3:=\iota_3(X^3)$, whose associated Dirichlet forms are denoted by $(\sE^+,\sF^+)$ and $(\sE^3,\sF^3)$ respectively. Take a constant $\tp>0$, and set a measure $\fm:=\tp\cdot \fm_+$ on $\fR_+$ and $\fm:=\fm_3$ on $\fR^3$, where $\fm_+:=m_+\circ \iota^{-1}_+$ and $\fm_3:=\left(\psi_\gamma(x)^2dx\right)\circ \iota^{-1}_3$. Then the desired dBMVD (with parameter $\tp$) is by definition an $\fm$-symmetric irreducible diffusion, denoted by $M:=(M_t)_{t\geq 0}$, of no killing at the origin $\b0$ such that
\begin{description}
\item{\rm (i)} The part process of $M$ on $\fR_+\setminus \{\b0\}$ is equivalent to that of $M^+$;
\item{\rm (ii)} The part process of $M$ on $\fR^3\setminus \{\b0\}$ is equivalent to that of $M^3$. 
\end{description}
The uniqueness of dBMVD in law is a consequence of \cite[Theorem~7.5.4]{CF}, as will be explained below its rigorous definition, i.e. Definition~\ref{definition-dBMVD}. The existence is concluded in Theorem~\ref{THM2} by piecing together $(\sE^+,\sF^+)$ and $(\sE^3,\sF^3)$ at the origin. More precisely, $M$ is indeed associated with
\[
\begin{aligned}
&\sF:=\left\{ f\in L^2(E,\fm): f|_{\fR_+}\in \sF^+, f|_{\fR^3}\in \sF^3, {}^+\widetilde{f|_{\fR_+}}(\b0)={}^3\widetilde{f|_{\fR^3}}(\b0) \right\},\\
&\sE(f,g):=\tp\cdot \sE^+(f|_{\fR_+}, g|_{\fR_+})+\sE^3(f|_{\fR^3}, g|_{\fR^3}),\quad f,g\in \sF,
\end{aligned}
\]
where ${}^+\tilde{f}$ (resp. ${}^3\tilde{f}$) stands for the $\sE^+$-quasi-continuous (resp. $\sE^3$-quasi-continuous) version of $f$. The dBMVD $M$ is shown to be a very nice diffusion as can reach everywhere and never explode in finite time, but the singularity happens at the origin. At a heuristic level, $M$ runs like $M^+$ or $M^3$ out of the origin, and once upon hitting $\b0$, it will decide to reflect at $\b0$ or to pass through $\b0$ to the other part of $E$. The parameter $\tp$ plays the role of ``skew'' constant weighing the possibilities of these two choices. Appealing to \eqref{EQ4DYD}, a bigger $\tp$ brings $M$ more opportunities to pass through $\b0$ from $\fR^3$ to $\fR_+$.


One of our main purposes is to derive the short-time heat kernel estimate for $M$.  The existence of the heat kernel $p(t,x,y)$ of $M$ with respect to its symmetric measure $\fm$ will be phrased in Proposition~\ref{PRO5} by an argument involving classical potential theory. Then the main result Theorem~\ref{main-thm} concludes the joint continuity of $p(t,x,y)$ on $(0,\infty)\times E\times E$ and obtains its explicit estimate for all $0<t\leq T$ with a fixed constant $T>0$ by splitting into three cases: (i) $x,y\in \fR_+$; (ii) $x\in \fR_+$ and $y\in \fR^3$; (iii) $x,y\in \fR^3$. Note incidentally that the case $x\in \fR^3, y\in \fR_+$ can be turned to (ii) because of the symmetry of $p(t,x,y)$ in $x$ and $y$. The approach to this estimate is based on the method of signed radial process developed in \cite{CL}. The signed radial process of $M$ is defined by $Y_t:=u(M_t)$ where $u(x):=|x|$ for $x\in \fR^3$ and $u(x):=-|x|$ for $x\in \fR_+$. It turns out to be a dBM on $\bR$, which can be realized as a skew Brownian motion transformed through a drift perturbation (i.e. Girsanov transform) by appealing to its pathwise representation (or associated well-posed SDE) \eqref{EQ4DYD}. Then a standard but very careful consideration derives the joint continuity of the heat kernel $\wh p^Y$ of $Y$ with respect to the symmetric measure of related skew Brownian motion (not Lebesgue measure!) and obtains a Gaussian type estimate for $\wh p^Y$ in Lemma~\ref{short-time-HKE-radial}. With this crucial lemma at hand, we eventually accomplish the estimate of $p(t,x,y)$ by utilizing the rotational invariance of $M$: For most cases including that $M$ runs on $\fR_+$ and that $M$ goes through $\b0$, the transition density is a radial function depending on $|x|$ and $|y|$ only, and for the reminder case that $M$ runs on $\fR^3\setminus \{\b0\}$, we bring into play Theorem~\ref{THM3} as mentioned early. The expression of $p(t,x,y)$ is explicitly determined by that of $\wh p^Y$ as we see in \eqref{eq-expression-p-12} and \eqref{eq-bar-q}, and this also leads to the joint continuity of $p(t,x,y)$. 

Though we will concentrate on dBMVDs with the specific three-dimensional component $X^3$, there are far more relevant stochastic models on space with varying dimension, which are of great interest to us. For example, the two-dimensional analogue of \eqref{eq-def-cL} exists as discussed in \cite[Appendix~F]{AGH05} (but there are natural obstructions to our story when $d\geq 4$), and in principle, we can formulate the two-dimensional analogue of $X^3$ and build a new dBMVD on $E$ with $d=2$. The major difficulty is that, unlike $\psi_\gamma$ in \eqref{eq-def-density}, the function $\phi$ for this case is not explicitly expressed, so that we have much more to do for accomplishing this analogical model. To be more general, a framework of dBMs on $\bR^d$ in \cite{FO89} characterizes their skew product decomposition in terms of Dirichlet forms. Roughly speaking, when $\phi(x)=\hat{\phi}(|x|)$ is a radial function depending only on $|x|$, it is possible to express a dBM $X$ in polar coordinates system as $X_t=\varrho_t \vartheta_{A_t}$, where $\varrho:=(\varrho_t)_{t\geq 0}$ is a diffusion on $\bR_+$, $\vartheta:=(\vartheta_t)_{t\geq 0}$ is the spherical Brownian motion on $S^{d-1}:=\{x\in \bR^d: |x|=1\}$ and $A:=(A_t)_{t\geq 0}$ is a ``clock" depending on $\varrho$. Under certain conditions additional to $1/(\hat{\phi}(r)^2r^{d-1})\in L^1_\mathrm{loc}(\bR_+)$ (this is clearly satisfied by $\psi_\gamma$), $\varrho$ can reach $0$ in finite time, and hence $X$ can reach the origin as well. Using this $X$ as multi-dimensional building block, general dBMVDs on $E$ with arbitrary $d\geq 2$ are possibly built, and we hope to explore them in the future. Another more involved density function
\[
	\phi(x)=\sum_{n=1}^N c_n\cdot \frac{\re^{-\gamma_n |x-x_n|}}{|x-x_n|},
\]
where $N$ is an integer, $c_n$ is a positive constant, $\gamma_n\in \bR$ and $\{x_n: 1\leq n\leq N\}\subset \bR^3$, gives a dBM on $\bR^3$ with $N$ centers $x_n$ manifesting strong attraction like the origin for $X^3$. We may tie a half line at each $x_n$ and get a new space with varying dimension. The study of dBMVDs on this new space is definitely an interesting plan of our future work as well. 

The rest of this paper is organized as follows. In \S\ref{SEC2-1} we introduce the dBM $X^3$ on $\bR^3$ with a parameter $\gamma\in \bR$. It associated Dirichlet form is formulated in Theorem~\ref{THM1}, and its global properties are utilized in \S\ref{SEC22} to observe a phase transition of related polymer model. Then the subsections \S\ref{SEC23} and \S\ref{SEC24} provide two characterizations of $X^3$ via $h$-transform and skew-product decomposition respectively. The short section \S\ref{SEC25} derives the pathwise representation of $X^3$ by means of Fukushima's decomposition, and particularly concludes that $X^3$ is not a semimartingale. 
The section \S\ref{SEC2-2} is devoted to the basic description of dBMVD $M$ on $E$. Its Dirichlet form characterization is stated in Theorem~\ref{THM2}, and some basic facts for it are summarized in the reminder subsections. Particularly, the existence of the transition density $p(t,x,y)$ of $M$ is obtained in Proposition~\ref{PRO5}. From \S\ref{SEC4}, we turn to the short-time heat kernel estimate for $M$. In \S\ref{SEC4} we study the signed radial process $Y$ of $M$. The associated Dirichlet form of $Y$ is obtained in Proposition~\ref{PRO6} and then we formulate its associated well-posed SDE in Theorem~\ref{THM6}, which shows that $Y$ is a skew Brownian motion transformed by a drift perturbation. The last section \S\ref{Sec-short-time-HKE} completes the short-time heat kernel estimate for $M$ by using a crucial lemma, i.e. Lemma~\ref{short-time-HKE-radial}, which obtains a Gaussian type heat kernel estimate for the signed radial process $Y$. 

Here are some notations for handy reference. 
We follow the convention that in the statements of the theorems or propositions $C, C_1, \cdots$ denote positive constants, whereas in their proofs $c, c_1, \cdots$ denote positive constants whose exact values are unimportant and may change from line to line.  In the meanwhile, the symbol $\lesssim$ (resp. $\gtrsim$) means that the left (resp. right) term is bounded by the right (resp. left) term multiplying a non-essential constant. Let $D\subset \bR^d$ be a domain. Then $C_c^\infty(D)$ is the family of all smooth functions with compact support in $D$. We use $\b0$ to denote the origin of $E$ (i.e. the origin of $\bR^4$), but the origin of $\bR^3$ or $\bR$  is simply denoted by $0$. The notation $|\cdot|$ stands for the Euclidean distance on $\bR^d$.

\section{Distorted Brownian motions on $\mathbb{R}^3$}\label{SEC2-1}

In this section, we give a brief overview for a special family of three-dimensional dBMs, which will be used as building blocks to construct dBMVDs in \S\ref{SEC2-2}. Fix a constant $\gamma\in \bR$ and set
\begin{equation}\label{EQ1PGX}
\psi_\gamma(x):= \frac{\mathrm{e}^{-\gamma|x|}}{2\pi|x|},\quad x\in \mathbb{R}^3.
\end{equation}
Further set a measure $m_\gamma(dx):=\psi_\gamma(x)^2dx$ on $\bR^3$. It is easy to verify that $m_\gamma$ is a positive Radon measure with full support, and $m_\gamma$ is finite, i.e. $\psi_\gamma \in L^2(\bR^3)$, if and only if $\gamma>0$. Define an energy form on $L^2(\mathbb{R}^3,m_\gamma)$ as follows:
\[
\begin{aligned}
\mathcal{F}^3&:=\left\{f\in L^2(\mathbb{R}^3,m_\gamma): \nabla f\in L^2(\mathbb{R}^3,m_\gamma)\right\}, \\
\mathcal{E}^3(f,g)&:=\frac{1}{2}\int_{\mathbb{R}^3}\nabla f(x)\cdot \nabla g(x)m_\gamma(dx),\quad f,g\in \mathcal{F}^3,
\end{aligned}
\]
where $\nabla f$ stands for the weak derivative of $f$. We shall write $(\mathcal{E}^{3,\gamma},\sF^{3,\gamma})$ for $(\mathcal{E}^3,\sF^3)$ when there is a risk of ambiguity.

\subsection{Associated dBM}

Some basic facts about $(\mathcal{E}^3,\mathcal{F}^3)$ are collected in the following theorem. Since  $(\mathcal{E}^3,\mathcal{F}^3)$ is indicated to be a regular Dirichlet form, we denote its associated dBM  by $X^3=\{(X^3_t)_{t\geq 0}, \left(\mathbf{P}_x^3\right)_{x\in\bR^3}\}$ henceforth. Note that the case $\gamma>0$ has been considered in \cite{FL17} but we present a different (and simpler) proof as below.

\begin{theorem}\label{THM1}
The following statements hold:
\begin{itemize}
\item[(i)] $(\mathcal{E}^3,\mathcal{F}^3)$ is a regular and irreducible Dirichlet form on $L^2(\mathbb{R}^3, m_\gamma)$ with $C_c^\infty(\mathbb{R}^3)$ being its special standard core. 
\item[(ii)] When $\gamma\geq 0$, $(\mathcal{E}^3,\mathcal{F}^3)$ is recurrent. When $\gamma<0$, it is transient. 
\end{itemize}
\end{theorem}
\begin{proof}
Recall that $(\mathcal{E}^{3, \gamma},\mathcal{F}^{3,\gamma})$ also denotes the quadratic form $(\mathcal{E}^3,\mathcal{F}^3)$.
It is straightforward to verify that $(\mathcal{E}^{3,\gamma},\mathcal{F}^{3,\gamma})$ is a Dirichlet form on $L^2(\mathbb{R}^3, m_\gamma)$ and $C_c^\infty(\bR^3)\subset \mathcal{F}^{3,\gamma}$. The irreducibility of $(\mathcal{E}^{3, \gamma},\mathcal{F}^{3,\gamma})$ for the case $\gamma>0$ has been proved in \cite[Proposition~2.4]{FL17}. For the case $\gamma\leq 0$ it can be concluded by the comparison of irreduciblity presented in \cite[Corollary~4.6.4]{FOT}. To show $C_c^\infty(\bR^3)$ is $\mathcal{E}^{3,\gamma}_1$-dense in $\mathcal{F}^{3,\gamma}$, we first note that this is true for the case $\gamma=0$ since $\psi_0$ belongs to the so-called Muckenhoupt's class; see e.g. \cite{K94}. Then it suffices to consider the case $\gamma\neq 0$. Let $\mathcal{F}^{3,\gamma}_\rc$ (resp. $\mathcal{F}^{3,0}_\rc$) be the family of all bounded functions with compact support in $\mathcal{F}^{3,\gamma}$ (resp. $\mathcal{F}^{3,0}$). We first assert that $\mathcal{F}^{3,\gamma}_\rc$ is $\mathcal{E}^{3,\gamma}_1$-dense in  $\mathcal{F}^{3,\gamma}$. To do this, fix $f\in \mathcal{F}^{3,\gamma}$ and assume without loss of generality that $f$ is bounded (see \cite[Theorem~1.4.2]{FOT}). Take  $\eta\in C_c^\infty(\bR^3)$ such that $0\leq \eta\leq 1$ and $\eta\equiv 1$ on $\{x: |x|\leq 1\}$. Set $\eta_n(x):=\eta(x/n)$ and $f_n:=f\cdot \eta_n\in \mathcal{F}^{3,\gamma}_\rc$ for all $n\in \bN$. Since $0\leq \eta_n\leq 1$ and $\eta_n\rightarrow 1$ pointwisely,  it follows from the dominated convergence theorem that 
\[
	\int |f-f_n|^2dm_\gamma=\int |f|^2\cdot |1-\eta_n|^2 dm_\gamma \rightarrow 0
\] 
as $n\uparrow \infty$. On the other hand, 
\[
	\nabla f-\nabla f_n =(1-\eta_n)\nabla f -\frac{f}{n} \nabla \eta\left(\frac{x}{n}\right).
\]
Since $\nabla f, f\in L^2(\bR^3,m_\gamma)$ and $\nabla \eta$ is bounded, one can obtain that 
\[
\int |\nabla f- \nabla f_n|^2dm_\gamma\rightarrow 0. 
\]
Hence $\mathcal{E}^{3,\gamma}_1(f_n-f,f_n-f)\rightarrow 0$ as $n\rightarrow \infty$. Next fix $g\in \mathcal{F}^{3,\gamma}_\rc$ and take $L>0$ such that $\text{supp}[g]\subset \{x: |x|<L\}$. Clearly, there exist two appropriate positive constants $c_1$ and $c_2$ (depending on $L$ and $\gamma$) such that for all $x$ with $|x|<L$, 
\begin{equation}\label{EQ2CPX}
	c_1\psi_0(x)\leq \psi_\gamma(x)\leq c_2\psi_0(x). 
\end{equation}
This implies $g\in \mathcal{F}^{3,0}_\rc$. Then there exists a sequence of functions $g_n\in C_c^\infty(\bR^3)$ with $\text{supp}[g_n]\subset \{x: |x|<L\}$ converging to $g$ relative to the $\mathcal{E}^{3,0}_1$-norm. By using \eqref{EQ2CPX} again, we can conclude that $g_n$ also converges to $g$ relative to the $\mathcal{E}^{3,\gamma}_1$-norm. Therefore,  $(\mathcal{E}^{3, \gamma},\mathcal{F}^{3,\gamma})$ is regular and $C_c^\infty(\bR^3)$ is a special standard core of it. 

The recurrence of $(\mathcal{E}^{3, \gamma},\mathcal{F}^{3,\gamma})$ for the case $\gamma>0$ has been already illustrated in \cite[Proposition~2.4]{FL17}. For the case $\gamma=0$, let $\eta_n$ be as above. Then $\eta_n\in \mathcal{F}^{3,0}$ and $\eta_n\rightarrow 1$ pointwisely. To obtain the recurrence of $(\mathcal{E}^{3, 0},\mathcal{F}^{3,0})$, it suffices to show $\mathcal{E}^{3,0}(\eta_n,\eta_n)\rightarrow 0$ as $n\rightarrow \infty$. Indeed, 
\[
	\mathcal{E}^{3,0}(\eta_n,\eta_n)=\frac{1}{2n^2}\int \left|(\nabla \eta) \left(x/n \right)  \right|^2\frac{dx}{|x|^2}=\frac{1}{2n}\int \left|\nabla \eta (x)\right|^2\frac{dx}{|x|^2}\rightarrow 0. 
\]
Finally consider the case $\gamma<0$. Since $\psi_\gamma(x)\geq |\gamma|/(2\pi)$ for all $x$, it follows that for all $f\in C_c^\infty(\bR^3)$,
\[
	\mathcal{E}^{3,\gamma}(f,f)\geq \frac{\gamma^2}{8\pi^2}\int |\nabla f|^2dx =: \frac{\gamma^2}{8\pi^2} \mathbf{D}(f,f), 
\]
where $\mathbf{D}$ is the Dirichlet integral that induces the associated Dirichlet form of three-dimensional Brownian motion. Clearly, three-dimensional Brownian motion is transient. By virtue of \cite[Theorem~1.6.4]{FOT}, we can conclude the transience of $(\mathcal{E}^{3,\gamma}, \mathcal{F}^{3,\gamma})$. This completes the proof. 
\qed \end{proof}

When $\gamma\geq 0$, $(\mathcal{E}^3,\mathcal{F}^3)$ is not only recurrent but also ergodic in the following sense: For $\gamma>0$ and any $x\in \bR^3$,
\begin{equation}\label{EQ2TTP}
	\frac{1}{t}\int_0^t \mathbf{P}_x^3(X^3_s\in \cdot)ds \rightarrow \frac{m_\gamma(\cdot)}{m_\gamma(\bR^3)}=2\pi\gamma m_\gamma(\cdot),\quad \text{weakly as }t\uparrow \infty.
\end{equation}
For $\gamma=0$, the probability measure on the left hand side is vaguely convergent to $0$ as $t\uparrow \infty$. See e.g. \cite[Theorem~4.7.3]{FOT}. 

\subsection{Generator and motivated polymer model}\label{SEC22}

The dBM $X^3$ is motivated by a singular polymer model explored in e.g. \cite{CKMV2, CKMV1, CM}. Let us use a few lines to explain it. Fix $T>0$ and let $\Omega_T:=C([0,T], \mathbb{R}^d)$, i.e. the family of all continuous paths of size $T$ in $\bR^d$, be the configuration space of the system. Then the polymer model is described by a Gibbs ensemble at each inverse temperature $\beta$ ($\geq 0$), realized as a probability measure $\mathbf{Q}_{\beta, T}$ on $\Omega_T$, which is also called a Gibbs measure. More precisely, the underlying probability measure $\bfQ_{0,T}$ is identified with the Wiener measure on $\Omega_T$ in this model, and we also denote it by $\mathbf{Q}_T$ in abbreviation. For $\beta>0$, $\mathbf{Q}_{\beta,T}$ is determined by the so-called Hamiltonian $H_T$, which is given by a certain potential function $v$ on $\mathbb{R}^d$ in the following manner:
\begin{equation}\label{EQ1HOT}
	H_T(\omega)=-\int_0^T v(\omega(t))dt,\quad \omega\in \Omega_T. 
\end{equation}
In other words, 
\begin{equation}\label{EQ1PTO}
	\mathbf{Q}_{\beta, T}(d\omega)=\frac{\exp\left(-\beta H_T(\omega)\right)}{Z_{\beta,T}}\mathbf{Q}_T(d\omega) =\frac{\exp\left(\beta \int_0^Tv(\omega(t))dt\right)}{Z_{\beta,T}}\mathbf{Q}_T(d\omega),
\end{equation}
where $Z_{\beta,T}:=\mathbf{E}_T\exp\left(-\beta H_T\right)$ is the so-called partition function. The motivated model is with dimension $3$, i.e. $d=3$, and given by a singular potential function $v=\delta_0$, i.e. the delta function at the origin. In this case, the Hamiltonian $H_T$ is understood as a limitation $-\lim_{\varepsilon \downarrow 0} \int_0^T A_\varepsilon \cdot 1_{(-\varepsilon,\varepsilon)}(\omega_t)dt$ in a certain manner, where $A_\varepsilon$ ($\uparrow \infty$ as $\varepsilon\downarrow 0$) is a constant depending on a crucial parameter $\gamma\in \bR$, and meanwhile at a heuristic level the inverse temperature $\beta$ in \eqref{EQ1PTO} is also retaken to be a function of $\gamma$, i.e. $\beta:=\beta_\gamma> 0$ (and $\beta_{-\infty}:=0$); see e.g. \cite{CKMV1}.
There are at least three ways to manifest the phase transition parametrized by $\gamma$ and the critical value is $\gamma_{cr}=0$ --- The first two are already mentioned in \cite{CKMV1} and the last one is due to the Dirichlet form characterization of $X^3$:

\begin{description}
\item[(1)] The first way is to observe the thermodynamic limit of $\bfQ_{\beta_\gamma, T}$ as $T\uparrow \infty$. It can be shown that (see e.g. \cite{CKMV1}) 
	\begin{itemize}
	\item[(i)] When $\gamma>\gamma_{cr}=0$, the limiting measure of $\bfQ_{\beta_\gamma, T}$ under suitable scaling as $T\uparrow \infty$ exists and induces a diffusion process on $\bR^3$. In fact, this process is nothing but $X^3$ obtained in Theorem~\ref{THM1}, which possesses an ergodic distribution $2\pi\gamma m_\gamma(dx)=2\pi\gamma\psi_\gamma(x)^2dx$ (see \eqref{EQ2TTP}). In this case, the ensemble is called in the globular state;
	\item[(ii)] When $\gamma=\gamma_{cr}=0$, the limiting process is mixed Gaussian; 
	\item[(iii)] When $\gamma<\gamma_{cr}=0$, the scaling is taken to be a different one, and the limiting process is nothing but three-dimensional Brownian motion. In this case, the ensemble is called in the diffusive state. 
	\end{itemize}
\item[(2)] The second way is to analyse the spectrum of the informal Schr\"odinger operator
\begin{equation}\label{EQ2DBG}
	\frac{1}{2}\Delta+\beta_{\gamma}\cdot \delta_0: L^2(\bR^3)\rightarrow L^2(\bR^3),
\end{equation}
where $\Delta$ is the Laplacian. 
Note that all self-adjoint extensions on $L^2(\bR^3)$ of $\frac{1}{2}\Delta$ restricting to $C_c^\infty(\bR^3\setminus \{0\})$ can be parametrized by a constant $\gamma\in \{-\infty\}\cup \bR$; see e.g. \cite[Theorem~2.1]{CKMV1}. Denote the family of all these extensions by $\{\cL_\gamma: \gamma=-\infty\text{ or }\gamma\in \bR\}$ and particularly $\cL_{-\infty}=\frac{1}{2}\Delta$ corresponds to the underlying case (Recall that $\beta_{-\infty}=0$). Then \eqref{EQ2DBG} should be understood as $\cL_\gamma$ in a rigorous sense. Denote the spectrum set of $\cL_\gamma$ by $\sigma(\cL_\gamma)$. It is well known that
\begin{itemize}
\item[(i)] When $\gamma>\gamma_{cr}=0$, $\sigma(\cL_\gamma)=(-\infty, 0]\cup \{\gamma^2/2\}$. Moreover, $\psi_\gamma$ is the ground state of $\cL_\gamma$, i.e. 
\[
	\cL_\gamma \psi_\gamma=\frac{\gamma^2}{2}\psi_\gamma,
\]
and $\gamma^2/2$ is exactly the free energy of the ensemble, i.e. 
\[
	\lim_{T\uparrow \infty} \frac{\log Z_{\beta_\gamma, T}}{T}=\frac{\gamma^2}{2};
\]
\item[(ii)] When $\gamma\leq \gamma_{cr}= 0$, $\sigma(\cL_\gamma)=(-\infty, 0]$ and no eigenvalues exist. 
\end{itemize}
\end{description}

The third way is based on the relation between $\cL_\gamma$ and the generator of $(\mathcal{E}^3,\mathcal{F}^3)$. Define an operator $\cA_\gamma$ on $L^3(\bR^3,m_\gamma)$ by an informal $h$-transform as follows:
\begin{equation}\label{generator-3d-dBM}
\begin{aligned}
&\mathcal{D}(\cA_\gamma)=\{f\in L^2(\mathbb{R}^3, m_\gamma): f \psi_\gamma\in \mathcal{D}(\mathcal{L}_\gamma)\}, \\
&\cA_\gamma f=\frac{\mathcal{L}_\gamma(\psi_\gamma f)}{\psi_\gamma}-\frac{\gamma^2}{2} f,\qquad f\in \mathcal{D}(\cA_\gamma). 
\end{aligned}
\end{equation}
It is not hard to verify that $C_c^\infty(\bR^3\setminus \{0\})\subset \cD(\cA_\gamma)$ and for $f\in C_c^\infty(\bR^3\setminus \{0\})$ (see e.g. \cite[(2.3)]{FL17}), 
\[
	\cA_\gamma f=\frac{1}{2}\Delta f +\frac{\nabla \psi_\gamma}{\psi_\gamma}\cdot \nabla f. 
\]
The lemma below links $(\mathcal{E}^3,\mathcal{F}^3)$ with $\cL_\gamma$. 

\begin{lemma}\label{LM1}
The operator $\cA_\gamma$ defined by \eqref{generator-3d-dBM} is the generator of $(\mathcal{E}^3,\mathcal{F}^3)$. 
\end{lemma}
\begin{proof}
The case $\gamma>0$ has been shown in \cite{FL17}. For the case $\gamma\leq 0$, see \cite[Appendix F]{AGH05}.
\qed \end{proof}
\begin{remark}\label{RM1}
Since the semigroup associated with $\cL_\gamma$ admits a symmetric transition density with respect to the Lebesgue measure, i.e. there exists a suitable function $r_t^\gamma(x,y)$ such that $r_t^\gamma(x,y)=r_t^\gamma(y,x)$ and $R^\gamma_tf(x):=\int_{\bR^3}r^\gamma_t(x,y)f(y)dy$ forms this semigroup (see e.g. \cite{CKMV1}), this lemma tells us the semigroup associated with $\cA_\gamma$ admits a symmetric transition density with respect to $m_\gamma$:
\[
	p^\gamma_t(x,y):=\frac{\mathrm{e}^{-\frac{\gamma^2}{2} t}\cdot r^\gamma_t(x,y)}{\psi_\gamma(x)\psi_\gamma(y)}.
\]
In other words, 
\[
	P^\gamma_tf(x):=\int_{\bR^3}p^\gamma_t(x,y)f(y)m_\gamma(dy)=\int_{\bR^3}\frac{\mathrm{e}^{-\frac{\gamma^2}{2} t}\psi_\gamma(y) r^\gamma_t(x,y)}{\psi_\gamma(x)}f(y)dy,\quad f\in L^2(\bR^3,m_\gamma),
\]
is the semigroup associated with $\cA_\gamma$. 
\end{remark}

Then Theorem~\ref{THM1} leads to the third reflection of the same phase transition:
\begin{description}
\item[(3)] Under the transform \eqref{generator-3d-dBM}, $\cL_\gamma$ corresponds to the dBM $X^3$. The global property of $X^3$ depending on $\gamma$ manifests the same phase transition as the two mentioned above: When $\gamma\geq \gamma_{cr}=0$, $X^3$ is recurrent; otherwise it is transient. The difference between the diffusive state $\gamma>0$ and the critical state $\gamma=0$ has already illustrated by the ergodicity of $X^3$ after the proof of Theorem~\ref{THM1}. 
\end{description}

\begin{remark}
A similar discussion about the critical phenomenon of certain Markovian Schr\"odinger forms appeared in \cite{T14}, where $h$-transform and global properties of Dirichlet forms are employed as well. However in the current paper, the Schr\"odinger form induced by $\cL_\gamma$ (or informally by \eqref{EQ2DBG}) is not Markovian. In other words, $\cL_\gamma$ can not be the generator of a certain Markov process.
\end{remark}

\subsection{Characterization via $h$-transform}\label{SEC23}

This subsection is devoted to illustrating some connections between $X^3$ and three-dimensional Brownian motion. We use the notation $R_t:=R^{-\infty}_t$ to stand for the probability transition semigroup of three-dimensional Brownian motion $W=(W_t)_{t\geq 0}$ as well as its $L^2$-semigroup if no confusions caused. Note that $\psi_\gamma$ is finite out of $\{0\}$. Consider the following $h$-transform with $h:=\psi_\gamma$: 
\begin{equation}\label{EQ5HPG}
	{}_h R^{\gamma}_t(x,dy):=\left\lbrace
\begin{aligned}
& \mathrm{e}^{-\frac{\gamma^2}{2} t}\frac{\psi_\gamma(y)}{\psi_\gamma(x)}R_t(x,dy),\quad x\in \bR^3\setminus \{0\},\\
&0,\qquad \qquad \quad \quad\quad\quad\;\;\;\;\; x=0. 
\end{aligned}
\right. 
\end{equation}
It is not hard to verify that $\psi_\gamma$ is $\frac{\gamma^2}{2}$-excessive relative to $R_t$ in the sense that $\mathrm{e}^{-\gamma^2t/2}R_t\psi_\gamma\leq \psi_\gamma$ and ${}_hR^\gamma_t$ is a sub-Markovian semigroup on $\bR^3\setminus \{0\}$. Denote the induced Markov process of ${}_hR^\gamma_t$ on $\bR^3\setminus \{0\}$ by ${}_hW^\gamma=\{({}_hW^\gamma_t)_{t\geq 0},\left({}_h\mathbf{P}^\gamma_x\right)_{x\in \bR^3}, \zeta_h\}$, where ${}_h\mathbf{P}^\gamma_x$ is the law of ${}_hW^{\gamma}$ starting from $x$ and $\zeta_h$ is its life time.

\begin{remark}
When $\gamma\geq 0$, $\psi_\gamma$ is nothing but the $\gamma^2/2$-resolvent kernel of $W$. More precisely, let $r(t,x)$ be the Gaussian kernel, i.e. $r(t,x)=\frac{1}{(2\pi t)^{3/2}}\mathrm{e}^{-\frac{|x|^2}{2t}}$. Then 
\[
	\psi_\gamma(x)=\int_0^\infty \mathrm{e}^{-\frac{\gamma^2 t}{2}}r(t,x)dt. 
\]
Particularly $\psi_0$ coincides with the three-dimensional Newtonian potential kernel. 
\end{remark}

To phrase an alternative characterization of $X^3$, we prepare two notions. Let $E$ be a locally compact separable metric space and $\fm$ be a positive Radon measure on it. 
The first one is the so-called \emph{part process}; see \cite[\S4.4]{FOT}. Let $(\mathcal{E},\mathcal{F})$ be a Dirichlet form on $L^2(E,\fm)$ associated with a Markov process $X$  and $F\subset E$ be a closed set of positive capacity relative to $(\mathcal{E},\mathcal{F})$. Then the part process $X^G$ of $X$ on $G:=E\setminus F$  is obtained by killing $X$ once upon leaving $G$. In other words, 
\[
	X^G_t=\left\lbrace
	\begin{aligned}
	& X_t,\quad t<\sigma_F:=\{s>0: X_s\in F\}, \\
	& \partial, \quad\;\; t\geq \sigma_F,
	\end{aligned}
	\right. 
\]
where $\partial$ is the trap of $X^G$. Note that $X^G$ is associated with the \emph{part Dirichlet form} $(\mathcal{E}^G,\mathcal{F}^G)$ of $(\mathcal{E},\mathcal{F})$ on $G$: 
\begin{equation}\label{EQ5FGF}
\begin{aligned}
	&\mathcal{F}^G=\{f\in \mathcal{F}: \tilde{f}=0,\; \mathcal{E}\text{-q.e. on }F\}, \\
	&\mathcal{E}^G(f,g)=\mathcal{E}(f,g),\quad f,g\in \mathcal{F}^G,
\end{aligned}
\end{equation}
where $\tilde{f}$ stands for the $\cE$-quasi-continuous version of $f$. The second is the one-point reflection of a Markov process studied in \cite{CF15}; see also \cite[\S7.5]{CF}. Let $a\in E$ be a non-isolated point with $\fm(\{a\})=0$ and $X^0$ be an $\fm$-symmetric Borel standard process on $E_0:=E\setminus \{a\}$ with no killing inside. Then a right process $X$ on $E$ is called a \emph{one-point reflection} of $X^0$ (at $a$) if $X$ is $\fm$-symmetric and of no killing on $\{a\}$, and the part process of $X$ on $E_0$ is $X^0$.

\begin{theorem}\label{THM3}
Fix $\gamma\in \bR$, and let $X^3$ and $(\mathcal{E}^3,\mathcal{F}^3)$ be in Theorem~\ref{THM1}. Then $\{0\}$ is of positive $1$-capacity relative to $\mathcal{E}^3$. Furthermore, the following hold:
\begin{itemize}
\item[(1)] ${}_hW^\gamma$ is identified with the part process of $X^3$ on $\bR^3\setminus \{0\}$;
\item[(2)] $X^3$ is the unique (in law) one-point reflection of ${}_hW^\gamma$ at $0$. 
\end{itemize}
\end{theorem}
\begin{proof}
To show the $1$-capacity of $\{0\}$ is positive, the case $\gamma>0$ has been considered in \cite[Proposition~3.1]{FL17}. Denote the $1$-capacity relative to $\mathcal{E}^{3,1}$ by $\text{Cap}^1$, and then $\text{Cap}^1(\{0\})>0$. For the case $\gamma\leq 0$, denote the $1$-capacity relative to $\mathcal{E}^{3,\gamma}$ by $\text{Cap}^\gamma$. It suffices to note that $\text{Cap}^\gamma(A)\geq \text{Cap}^1(A)$ for any Borel set $A\subset \bR^3$ due to $\mathcal{F}^{3,\gamma}\subset\mathcal{F}^{3,1}$ and $\mathcal{E}^{3,1}(f,f)\leq \mathcal{E}^{3,\gamma}(f,f)$ for all $f\in \mathcal{F}^{3,\gamma}$. Particularly, $\text{Cap}^\gamma(\{0\})\geq \text{Cap}^1(\{0\})>0$. 

Denote the Dirichlet form of three-dimensional Brownian motion by $(\frac{1}{2}\mathbf{D}, H^1(\bR^3))$, i.e. $H^1(\bR^3)$ is the $1$-Sobolev space and $\mathbf{D}$ is the Dirichlet integral. To prove the first assertion, it is straightforward to verify that $({}_hR^{\gamma}_t)$ is symmetric with respect to $m_\gamma(dx)=\psi_\gamma(x)^2dx$ and then associated with the Dirichlet form (see \cite[(1.3.17)]{FOT})
\[
\begin{aligned}
{}_h\mathcal{F}&=\{f\in L^2(\bR^3,m_\gamma): {}_h\mathcal{E}(f,f)<\infty\}, \\
{}_h\mathcal{E}(f,g)&=\lim_{t\downarrow 0}\frac{1}{t}\int_{\bR^3} \left(f(x)-{}_hR^{\gamma}_tf(x)\right)g(x)m_\gamma(dx),\quad f,g\in {}_h\mathcal{F}. 
\end{aligned}
\]
One can easily deduce that for any $f\in L^2(\bR^3,m_\gamma)$, 
\[
\begin{aligned}
	{}_h\mathcal{E}(f,f)&=\lim_{t\downarrow 0}\frac{1}{t}\int_{\bR^3} \left(f(x)\psi_\gamma(x)-\mathrm{e}^{-\frac{\gamma^2}{2} t}R_t(f\psi_\gamma)(x)\right)(f\psi_\gamma)(x)dx \\
	&=\frac{1}{2}\int_{\bR^3} |\nabla (f\psi_\gamma)|^2(x)dx+\frac{\gamma^2}{2}\int_{\bR^3} |(f\psi_\gamma)|^2(x)dx \\
	&=\frac{1}{2}\mathbf{D}_{\gamma^2}(f\psi_\gamma,f\psi_\gamma),
\end{aligned}\]
whenever the limit exists. 
This leads to
\begin{equation}\label{EQ5FFF}
	{}_h\mathcal{F}=\{f: f\psi_\gamma \in H^1(\bR^3)\},\quad {}_h\mathcal{E}(f,f)=\frac{1}{2}\mathbf{D}_{\gamma^2}(f\psi_\gamma, f\psi_\gamma),\quad f\in {}_h\mathcal{F}. 
\end{equation}
Since $C_c^\infty(\bR^3\setminus \{0\})$ is a core of $(\frac{1}{2}\mathbf{D},H^1(\bR^3))$ and $\psi_\gamma \in C^\infty(\bR^3\setminus \{0\})$ is positive, we can conclude that $C_c^\infty(\bR^3\setminus \{0\})$ is also a core of $({}_h\mathcal{E},{}_h\mathcal{F})$. 
On the other hand, the part process $X^{3,\bR^3\setminus \{0\}}$ of $X^{3}$ on $\bR^3\setminus \{0\}$ is associated with the Dirichlet form $(\mathcal{E}^{3,\bR^3\setminus \{0\}},\mathcal{F}^{3,\bR^3\setminus \{0\}})$ given by \eqref{EQ5FGF} with $(\mathcal{E},\mathcal{F})=(\mathcal{E}^{3},\mathcal{F}^{3})$ and $G=\bR^3\setminus \{0\}$. Particularly, $C_c^\infty(\bR^3\setminus \{0\})$ is also a core of $(\mathcal{E}^{3,\bR^3\setminus \{0\}},\mathcal{F}^{3,\bR^3\setminus \{0\}})$ by \cite[Theorem~4.4.3]{FOT}. It follows from Lemma~\ref{LM1} and $\cL_\gamma=\frac{1}{2}\Delta$ on $C_c^\infty(\bR^3\setminus \{0\})$ that for any $f\in C_c^\infty(\bR^3\setminus\{0\})\subset \cD(\cA_\gamma)$,
\[
\begin{aligned}
	\mathcal{E}^{3,\bR^3\setminus \{0\}}(f,f)&=\mathcal{E}^3(f,f)= \left(-\cA_\gamma f, f \right)_{m_\gamma}  \\
	&= -\int_{\bR^3} \cL_\gamma(\psi_\gamma f)(x)\left(\psi_\gamma f\right)(x)dx +\frac{\gamma^2}{2}\int_{\bR^3} |(f\psi_\gamma)|^2(x)dx \\
	&=\frac{1}{2}\mathbf{D}_{\gamma^2}(f\psi_\gamma, f\psi_\gamma). 
\end{aligned}\]
In view of \eqref{EQ5FFF}, one can obtain that 
\[
	\mathcal{E}^{3,\bR^3\setminus \{0\}}(f,f)={}_h\mathcal{E}(f,f),\quad \forall f\in C_c^\infty(\bR^3\setminus \{0\}). 
\]
As a result, $(\mathcal{E}^{3,\bR^3\setminus \{0\}},\mathcal{F}^{3,\bR^3\setminus \{0\}})=({}_h\mathcal{E},{}_h\mathcal{F})$. Therefore, ${}_hW^\gamma$ is identified with the part process of $X^3$ on $\bR^3\setminus \{0\}$. 

Finally we prove the second assertion. Clearly, $X^{(\gamma)}$ is a one-point reflection of ${}_hW^{\gamma}$ by the first assertion. To show the uniqueness, we shall apply \cite[Theorem~7.5.4]{CF}. It suffices to note that for every $x\neq 0$,
\begin{equation}\label{EQ5HPG2}
	{}_h\mathbf{P}^\gamma_x(\zeta_h<\infty, {}_hW^{\gamma}_{\zeta_h-}=0)={}_h\mathbf{P}^\gamma_x(\zeta_h<\infty)=\mathbf{P}^3_x(\sigma_0<\infty)>0,
\end{equation}
where $\sigma_0:=\inf\{t>0: X^{3}_t=0\}$. The second equality is due to the conservativeness of $X^3$ (see Corollary~\ref{COR1}), and the first one is the consequence of that ${}_hW^\gamma$ has no killing inside (on $\bR^3\setminus \{0\}$) and the quasi-left continuity of $X^3$. The last equality holds because $\{0\}$ is of positive capacity (or by virtue of Lemma~\ref{LM2}~(2)).  This completes the proof.
\qed \end{proof}
\begin{remark}
The second assertion in Theorem~\ref{THM3} leads to that $0$ is regular for itself with respect to $X^3$, i.e. $\bfP^3_0(\sigma_0=0)=1$. 
\end{remark}

\subsection{Characterization via skew product decomposition}\label{SEC24}

Due to the fact that $\psi_\gamma$ is a radial function, the part process ${}_hW^\gamma$ of $X^3$ on $\bR^3\setminus\{0\}$ is rotationally invariant in the following sense: Let $T$ be an arbitrary orthogonal transformation from $\bR^3$ to $\bR^3$, then 
\[
	{}_h\hat{W}^\gamma:=\left\{{}_h\hat{W}^\gamma_t:=T({}_hW^\gamma_t), {}_h\hat{\bfP}^\gamma_x:={}_h\bfP^\gamma_{T^{-1}x} \right\}
\]
defines an equivalent Markov process to ${}_hW^\gamma$. Hence we can characterize ${}_hW^\gamma$ by obtaining its skew product decomposition. Unsurprisingly, $X^3$ is also rotationally invariant (see \cite[pp.11]{FL17}) and it is not hard to figure out its radial process. The following lemma is an extension of \cite[Proposition~3.7]{FL17}, and the proof can be completed by the same argument (so we omit it). 

\begin{lemma}\label{LM2}
\begin{itemize}
\item[(1)] The process ${}_hW^\gamma$ admits a skew-product representation
\begin{equation}\label{EQ2HWG}
{}_hW^\gamma_t=\varrho^0_t \vartheta_{A^0_t},\quad t\geq 0,
\end{equation}
where $\varrho^0:=(\varrho^0_t)_{t\geq 0}=(|{}_hW^\gamma_t|)_{t\geq 0}$ is a symmetric diffusion on $(0,\infty)$, killed at $\{0\}$, whose speed measure $\ell_\gamma^0$ and scale function $\ts_\gamma^0$ are
\[
	\begin{aligned}
		&\ell_\gamma^0(dr)=\frac{\mathrm{e}^{-2\gamma r}}{\pi}dr,  \quad
		\ts_\gamma^0(r)=\left\lbrace \begin{aligned}
			&\frac{\pi}{2\gamma}\mathrm{e}^{2\gamma r},  \quad \text{when }\gamma \neq 0,\\
			&\pi r,\quad \text{when }\gamma=0,
		\end{aligned}\right. 
	\end{aligned}
	\quad r\in (0,\infty);
\]	
$A^0:=(A^0_t)_{t\geq 0}$ is the PCAF of $\varrho^0$ with the Revuz measure
\[
	\mu_{A^0}(dr):=\frac{\ell_\gamma^0(dr)}{r^2}
\]
and $\vartheta$ is a spherical Brownian motion on $S^2:=\{x\in \bR^3:|x|=1\}$, which is independent of $\varrho^0$.
\item[(2)] The radial process $\varrho=(\varrho_t)_{t\geq 0}:=(|X^3_t|)_{t\geq 0}$ is a symmetric diffusion on $[0,\infty)$, reflecting at $\{0\}$, whose speed measure $\ell_\gamma$ and scale function $\ts_\gamma$ are
\begin{equation}\label{EQ2LRE}
	\begin{aligned}
		&\ell_\gamma(dr)=\frac{\mathrm{e}^{-2\gamma r}}{\pi}dr,  \quad
		\ts_\gamma(r)=\left\lbrace \begin{aligned}
			&\frac{\pi}{2\gamma}\mathrm{e}^{2\gamma r},  \quad \text{when }\gamma \neq 0,\\
			&\pi r,\quad \text{when }\gamma=0,
		\end{aligned}\right. 
	\end{aligned}
	\quad r\in [0,\infty).
\end{equation}	
\end{itemize}
\end{lemma}

\begin{remark}
When $\gamma=0$, $\varrho^0$ is nothing but the absorbing Brownian motion on $(0,\infty)$ (killed at $0$), and $\varrho$ is the reflecting Brownian motion on $[0,\infty)$. It is not expected that $X^3$ admits an analogical representation of \eqref{EQ2HWG}, because $\ell_\gamma(dr)/r^2$ is not Radon on $[0,\infty)$ and hence not smooth (by e.g. \cite[Theorem~A.3.(4)]{L18}) relative to $\varrho$; see further explanation below \cite[Corollary~3.11]{FL17}. 
\end{remark}

It is worth noting two facts about the radial processes $\varrho^0$ and $\varrho$. The first one is to derive the global properties of $\varrho$, which lead to those of $X^3$, by employing the scale function and the speed measure.

\begin{corollary}\label{COR1}
Let $\varrho=(\varrho_t)_{t\geq 0}$ be the radial process of $X^3$. Then
\begin{itemize}
\item[(1)] $\varrho$ is irreducible and conservative. Particularly, $X^3$ is conservative. 
\item[(2)] $\varrho$ is recurrent, if and only if $\gamma\geq 0$. Otherwise it is transient. 
\end{itemize}
\end{corollary}
\begin{proof}
The irreducibility of $\varrho$ is clear. Note that $\varrho$ is conservative, if and only if (see e.g. \cite[Example~3.5.7]{CF})
\begin{equation}\label{EQ2LRT}
	\int_1^\infty \ell_\gamma((1,r))d\ts_\gamma(r)=\infty. 
\end{equation}
This is true by a straightforward computation. 

By \cite[Theorem~2.2.11]{CF}, $\varrho$ is transient, if and only if $\ts_\gamma(\infty):=\lim_{r\uparrow \infty} \ts_\gamma(r)<\infty$ and $\ell_\gamma((1,\infty))=\infty$ (Otherwise it is recurrent). Clearly it amounts to $\gamma<0$. This completes the proof. 
\qed \end{proof}
\begin{remark}
The recurrence/transience of $\varrho$ coincides with that of $X^3$, as stated in Theorem~\ref{THM1}. 
\end{remark}

The second is concerned with their pathwise decompositions as below. The proof is analogical to that of \cite[(3.6)]{FL17} and we omit it.

\begin{corollary}\label{COR2}
The radial processes $\varrho^0$ and $\varrho$ admit the following pathwise decompositions:
\[
\begin{aligned}
	& \varrho^0_t-\varrho^0_0=B_t-\gamma t,\quad 0\leq t<\zeta^0(=\sigma_0), \\
	& \varrho_t-\varrho_0=B_t-\gamma t+ \pi \gamma \cdot l^0_t,\quad t\geq 0, 
\end{aligned}
\]
where $\zeta^0$ is the lift time of $\varrho^0$, $(B_t)_{t\geq 0}$ is a certain one-dimensional standard Brownian motion, and $l^0:=(l^0_t)_{t\geq 0}$ is the local time of $\varrho$ at $0$, i.e. the PCAF associated with the smooth measure $\delta_0$ relative to $\varrho$. 
\end{corollary}

\subsection{Pathwise representation}\label{SEC25}

This short subsection is devoted to the pathwise representation of $X^3$ by virtue of so-called Fukushima's decomposition. Let $f_i$ be the $i$th coordinate function for $i=1,2,3$, i.e. $f_i(x):=x_i$ for $x=(x_1,x_2,x_3)\in \bR^3$. Obviously $f_i\in \mathcal{F}^3_{\mathrm{loc}}$. Then we can write the Fukushima's decomposition of $X^3$ relative to $f_i$: 
\begin{equation}\label{EQ2FIX}
	f_i(X^3_t)-f_i(X^3_0)=M^{f_i}_t+N^{f_i}_t,\quad t\geq 0,\quad \mathbf{P}^3_x\text{-a.s.},\quad \text{q.e. }x,
\end{equation}
where $M^{f_i}:=(M^{f_i}_t)_{t\geq 0}$ is an MAF locally of finite energy and $N^{f_i}:=(N^{f_i}_t)_{t\geq 0}$ is a CAF locally of zero energy. Note that $M^{f_i}$ and $N^{f_i}$ in this decomposition are unique in law. Set $M_t:=(M^{f_1}_t, M^{f_2}_t, M^{f_3}_t)$ and  $N_t:=(N^{f_1}_t, N^{f_2}_t, N^{f_3}_t)$. Recall that  
an additive functional $A=(A_t)_{t\geq 0}$ is called of bounded variation if $A_t(\omega)$ is of bounded variation in $t$ on each compact subinterval of $[0,\zeta(\omega))$ for every fixed $\omega$ in the defining set of $A$, where $\zeta$ ($=\infty$ for $X^3$ due to its conservativeness) is the life time of the underlying Markov process. We say $N:=(N_t)_{t\geq 0}$ is of bounded variation if $N^{f_i}$ is of bounded variation for $i=1,2,3$.

By repeating the arguments in \cite[\S4]{FL17}, we can conclude the following characterizations of $M$ and $N$.

\begin{theorem}
Let $X^3$ be in Theorem~\ref{THM1} and $M=(M_t)_{t\geq 0}, N=(N_t)_{t\geq 0}$ be in the Fukushima's decomposition \eqref{EQ2FIX}. Then the following hold:
\begin{itemize}
\item[(1)] For q.e. $x\in \bR^3$, $M$ is equivalent to a three-dimensional Brownian motion under $\bfP^3_x$. 
\item[(2)] For $t<\sigma_0$, 
\[
	N_t=-\int_0^t \frac{\gamma |X^3_s|+1}{|X^3_s|^2}\cdot X^3_s ds.
\]
However, $N$ is not of bounded variation. 
\end{itemize}
\end{theorem}

Comparing to the final property of $N$, the radial process $\varrho$ of $X^3$ is a semimartingale as presented in Corollary~\ref{COR2}. 
At a heuristic level, this behavior of $X^3$ is a consequence of the following fact: As noted by Erickson \cite{E90}, the excursions of $X^3$ away from $0$  oscillate so violently that each neighborhood of each point of the unit sphere is visited infinitely often by the angular part of $X^3$.



\section{Distorted Brownian motion  on space with varying dimension}\label{SEC2-2}

Let $$E:=\mathfrak{R}_+\cup \mathfrak{R}^3,$$ where $\fR_+:=\{(x_1,x_2,x_3,r)\in \bR^4: x_1=x_2=x_3=0, r\geq 0\}$($\simeq \bR_+:=[0,\infty)$) and $\mathfrak{R}^3:=\{(x_1,x_2,x_3,r)\in \bR^4: r=0\}$($\simeq \bR^3$), be the state space. For convenience, set the following maps:
\[
\iota_+: \mathbb{R}_+\rightarrow \fR_+, \quad r \mapsto (0,0,0,r),
\]
and
\begin{equation}\label{EQ3IRR}
	\iota_3: \mathbb{R}^3\rightarrow \fR^3, \quad (x_1,x_2,x_3)\mapsto (x_1,x_2,x_3,0).
\end{equation}
This section is devoted to the study of dBMVDs on $E$.

\subsection{One-dimensional part}\label{SEC31}

Let $\rho$ be a function on $\bR_+$ such that
\begin{equation}\label{EQ3RRR}
	\rho>0,\; \text{a.e.}, \quad \rho\text{ and } \frac{1}{\rho}\in L^1_\text{loc}(\bR_+),
\end{equation}
and
\begin{equation}\label{EQ3DRR}
\int_0^\infty \frac{dr}{\rho(r)}\int_0^r \rho(s)ds=\infty.
\end{equation}
Consider the Dirichlet form $(\cE^+,\cF^+)$ on $L^2(\bR_+,m_+):=L^2(\bR_+,\rho(r)dr)$:
\[
\begin{aligned}
	&\cF^+:=\left\{f\in L^2(\bR_+,m_+): f'\in L^2(\bR_+,m_+) \right\}, \\
	&\cE^+(f,g):=\frac{1}{2}\int_{\bR_+}f'(r)g'(r)m_+(dr),\quad f,g\in \cF^+. 
\end{aligned}
\]
The following lemma summarizes the basic facts about $(\cE^+,\cF^+)$. 

\begin{lemma}\label{LM3}
The following hold:
\begin{itemize}
\item[\rm (i)] $(\cE^+,\cF^+)$ is a regular strongly local Dirichlet form on $L^2(\bR_+,m_+)$ with a special standard core $C_c^\infty(\bR_+)$. It is also irreducible and conservative. For all $r\in \bR_+$, the singleton $\{r\}$ is of positive capacity relative to $\cE^+$. 
\item[\rm (ii)] The associated diffusion $X^+$ of $(\cE^+,\cF^+)$ is an irreducible conservative diffusion on $\bR_+$ reflecting at $0$, whose speed measure is $m_+$ and scale function is 
\[
	\ts_+(r)=\int_0^r\frac{1}{\rho(s)}ds,\quad r\geq 0. 
\] 
Furthermore, $X^+$ is transient (resp. recurrent), if and only if $1/\rho\in L^1(\bR_+)$ (resp. $1/\rho\notin L^1(\bR_+)$. 
\end{itemize}
\end{lemma}

The diffusion $X^+$ is usually called a \emph{distorted Brownian motion} on $\bR_+$; see e.g. \cite{L18}. 
The proof of Lemma~\ref{LM3} is referred to \cite[\S3.4]{LY172}. We should point out that $1/\rho\in L^1_\mathrm{loc}(\bR_+)$ implies the irreducibility of $(\cE^+,\cF^+)$ and the conservativeness is a consequence of \eqref{EQ3DRR}. The recurrence or transience of $X^+$ is indicated by \cite[Theorem~2.2.11]{CF}. 


\begin{example}\label{EXA1}
An interesting example is the radial process $\varrho=(\varrho_t)_{t\geq 0}$ of $X^3$ appearing in Lemma~\ref{LM2}. In this case, $\rho(r)=\re^{-2\gamma r}/\pi$ satisfies \eqref{EQ3RRR} and \eqref{EQ3DRR}.

Let $\hat{\varrho}$ be the diffusion on $\bR$ obtained by the symmetrization of $\varrho$. In other words, $\hat{\varrho}$ is associated with the energy form induced by the symmetric measure 
\[
	\hat{\ell}_\gamma(dr):=\frac{\mathrm{e}^{-2\gamma |r|}}{\pi}dr,\quad r\in \bR.
\]
With $\hat\phi_\gamma(r)=\mathrm{e}^{-\gamma |r|}/\sqrt{\pi}$ ($r\in \bR$) in place of $\psi_\gamma$, $\hat{\varrho}$ plays the same role as $X^3$ in the analogical one-dimensional model (parametrized by $\gamma$) of that explained in \S\ref{SEC22}. Particularly, under a similar $h$-transform to \eqref{generator-3d-dBM}, the generator of $\hat{\varrho}$ corresponds to a self-adjoint extension of $\frac{1}{2}\Delta$ restricting to $C_c^\infty(\bR\setminus \{0\})$. See e.g. \cite[Appendix F]{AGH05}.
\end{example}


 

\subsection{Definition}\label{SEC32}

Fix $\gamma\in \bR$ and a function $\rho$ satisfying \eqref{EQ3RRR} and \eqref{EQ3DRR} as above. Take a positive constant $\tp>0$. In this subsection, we rigorously give the definition for the so-called $(\rho, \gamma)$-dBMVD with the parameter $\tp$ on $E$. 

Recall that $E$ consists of two components $\fR_+$ and $\fR^3$. Roughly speaking, the distribution of such a dBMVD on $\fR_+$ (resp. $\fR^3$) is induced by the dBM $X^+$ (resp. $X^3$) in \S\ref{SEC31} (resp. Theorem~\ref{THM1}). To be more precise, set $M^+=(M^+_t)_{t\geq 0}:=(\iota_+(X^+_t))_{t\geq 0}$ and $M^3=(M^3_t)_{t\geq 0}:=(\iota_3(X^3_t))_{t\geq 0}$. Then $M^+$ is symmetric with respect to $\fm_+:=m_+\circ \iota_+^{-1}$ and associated with the Dirichlet form on $L^2(\fR_+, \fm_+)$:
\[
\begin{aligned}
	\sF^+&:=\{f: f\circ \iota_+\in \mathcal{F}^+\}, \\
	\sE^+(f,g)&:=\mathcal{E}^+(f\circ \iota_+, g\circ \iota_+),\quad f,g\in \sF^+. 
\end{aligned}
\]  
Accordingly, $M^3$ is symmetric with respect to $\fm_3:=m_\gamma\circ \iota_3^{-1}$ and associated with the Dirichlet form on $L^2(\fR^3,\fm_3)$:
\[
\begin{aligned}
	\sF^3&:=\{f: f\circ \iota_3\in \mathcal{F}^3\}, \\
	\sE^3(f,g)&:=\mathcal{E}^3(f\circ \iota_3, g\circ \iota_3),\quad f,g\in \sF^3. 
\end{aligned}
\]  
Define a measure $\fm$ on $E$ by $\fm|_{\fR_+}:=\tp\cdot\fm_+$ and $\fm|_{\fR^3}:=\fm_3$. Denote the density function by
\begin{equation}\label{EQ3HAG}
	h_{\rho, \gamma}(x):=\left\lbrace
		\begin{aligned}
			& \sqrt{\tp\cdot\rho(r)},\qquad\quad\; x=(0,0,0,r)\in \fR_+, \\
			&  \psi_\gamma((x_1,x_2,x_3)),\quad x=(x_1,x_2,x_3,0)\in \fR^3.
		\end{aligned}
	\right.
\end{equation} 
In other words, $\fm(dx)=h_{\rho,\gamma}(x)^2\mathtt{l}(dx)$, where $\mathtt{l}|_{\fR_+}$ is the one-dimensional Lebesgue measure and $\mathtt{l}|_{\fR^3}$ is the three-dimensional Lebesgue measure.  Then we introduce the following definition.

\begin{definition}[Distorted Brownian motion with varying dimension]\label{definition-dBMVD}
Fix $\tp>0$, $\gamma\in \bR$ and a function $\rho$ on $\bR_+$ satisfying \eqref{EQ3RRR} and \eqref{EQ3DRR}. A $(\rho,\gamma)$-distorted Brownian motion with varying dimension ($(\rho,\gamma)$-dBMVD in abbreviation) with parameter $\tp$ on $E$ is an $\fm$-symmetric irreducible diffusion $M=\{(M_t)_{t\geq 0}, \left(\mathbb{P}_x\right)_{x\in E}\}$ of no killing on $\{\b0\}$ such that 
\begin{description}
\item{\rm (i)} The part process of $M$ on $\fR_+\setminus \{\b0\}$ is equivalent to that of $M^+$;
\item{\rm (ii)} The part process of $M$ on $\fR^3\setminus \{\b0\}$ is equivalent to that of $M^3$. 
\end{description}
\end{definition}

This notion $(\rho,\gamma)$-dBMVD with parameter $\tp$ will be called dBMVD for short if no confusions caused. The uniqueness of dBMVD in law can be concluded by the following argument. Let $M^{+, \b0}$ (resp. $M^{3,\b0}$) be the part process of $M^+$ (resp. $M^3$) on $\fR_+\setminus \{\b0\}$ (resp. $\fR^3\setminus \{\b0\}$). Define a new Markov process $M^\b0$ on $E\setminus \{\b0\}$ by $M^\b0|_{\fR_+\setminus \{0\}}:=M^{+,\b0}$ and $M^\b0|_{\fR^3\setminus \{0\}}:=M^{3,\b0}$. Then the dBMVD $M$ is nothing but the one-point reflection of $M^\b0$ at $\b0$ by definition. The uniqueness of dBMVD in law is a consequence of \cite[Theorem~7.5.4]{CF}, Theorem~\ref{THM3} and Corollary~\ref{COR1}. 

\begin{remark}
It is worth noting that the parameter $\tp$ plays a role only in the symmetric measure $\fm$ (see the notes before Corollary~\ref{COR5}).  For different $\tp$, the dBMVDs are different as will be shown in Remark~\ref{RM12}, but their distributions out of the origin are exactly the same according to the definition.
\end{remark}

\subsection{Dirichlet form characterization of dBMVD}

 The main result of this section below gives the associated Dirichlet form of dBMVD. Recall that the Dirichlet forms $(\sE^+,\sF^+)$ and $(\sE^3,\sF^3)$ are given in \S\ref{SEC32}. Usually every function in a Dirichlet space is taken to be its quasi-continuous version tacitly. For $f\in \sF^+$ (resp. $f\in \sF^3$), the $\sE^+$-quasi-continuous (resp. $\sE^3$-quasi-continuous) version of $f$ will be denoted by ${}^+\tilde{f}$ (resp. ${}^3\tilde{f}$) when there is a risk of ambiguity. Since $\b0$ is of positive capacity relative to $\sE^+$ or $\sE^3$ due to Lemma~\ref{LM3} or Theorem~\ref{THM3}, ${}^+\tilde{f}$ or ${}^3\tilde{f}$ is well defined at $\b0$.

\begin{theorem}\label{THM2}
Let $(\sE^+,\sF^+)$ and $(\sE^3,\sF^3)$ be given in \S\ref{SEC32}. Then the quadratic form
\[
\begin{aligned}
&\sF:=\left\{ f\in L^2(E,\fm): f|_{\fR_+}\in \sF^+, f|_{\fR^3}\in \sF^3, {}^+\widetilde{f|_{\fR_+}}(\b0)={}^3\widetilde{f|_{\fR^3}}(\b0) \right\},\\
&\sE(f,g):=\tp\cdot \sE^+(f|_{\fR_+}, g|_{\fR_+})+\sE^3(f|_{\fR^3}, g|_{\fR^3}),\quad f,g\in \sF
\end{aligned}
\]
is a regular, strongly local and irreducible Dirichlet form on $L^2(E,\fm)$, whose associated Markov process is identified with the unique $(\rho,\gamma)$-dBMVD with parameter $\tp$ $M$.
\end{theorem}
\begin{proof}
Clearly, $(\sE,\sF)$ is a symmetric bilinear form satisfying the Markovian property. The strong locality of $(\sE,\sF)$ is indicated by that of $(\sE^+,\sF^+)$ and $(\sE^3,\sF^3)$. To show its closeness, take an $\sE_1$-Cauchy sequence $\{f_n:n\geq 1\}\subset \sF$. Then $f_n|_{\fR_+}$ is $\sE^+_1$-Cauchy and $f_n|_{\fR^3}$ is $\sE^3_1$-Cauchy. It follows from \cite[Theorem~2.1.4]{FOT} that there exists a subsequence $\{f_{n_k}:k\geq 1\}$ of $\{f_n\}$ and ${}^+f\in \sF^+, {}^3f\in \sF^3$ such that   
\[
\begin{aligned}
	&{}^+\widetilde{f_{n_k}|_{\fR_+}}\rightarrow {}^+ f,\quad \sE^+\text{-q.e.},  \\
	&{}^3\widetilde{f_{n_k}|_{\fR^3}}\rightarrow {}^3 f,\quad \sE^3\text{-q.e.}
\end{aligned}\]
and 
\[
	\tp\vee 1 \cdot \sE^+_1(f_n|_{\fR_+}-{}^+f,f_n|_{\fR_+}-{}^+f)+\sE^3_1(f_n|_{\fR^3}-{}^3f,f_n|_{\fR^3}-{}^3f)\rightarrow 0,\quad n\uparrow \infty. 
\]
Note that $\{\b0\}$ is of positive capacity relative to $\sE^+$ or $\sE^3$. This implies 
\[
	{}^+f(\b0)=\lim_{k\uparrow\infty}{}^+\widetilde{f_{n_k}|_{\fR_+}}(\b0)=\lim_{k\uparrow\infty}{}^3\widetilde{f_{n_k}|_{\fR^3}}(\b0)={}^3 f(\b0). 
\]
Hence the function $f$ given by $f|_{\fR_+}:={}^+f$ and $f|_{\fR^3}:={}^3f$ is well defined on $E$. In addition, $f\in \sF$ and
\[
	\sE_1(f_n-f,f_n-f)\leq \tp\vee 1\cdot \sE^+_1(f_n|_{\fR_+}-{}^+f,f_n|_{\fR_+}-{}^+f)+\sE^3_1(f_n|_{\fR^3}-{}^3f,f_n|_{\fR^3}-{}^3f)\rightarrow 0. 
\]
Therefore the closeness of $(\sE,\sF)$ is verified. 

Next, let us prove the regularity of $(\sE,\sF)$. Take a special standard core $\mathscr{C}^+$ of $(\sE^+,\sF^+)$ and a special standard core $\mathscr{C}^3$ of $(\sE^3,\sF^3)$; for example, $\sC^+:=C_c^\infty(\bR_+)\circ \iota_+^{-1}$ and $\sC^3:=C_c^\infty(\bR^3)\circ \iota_3^{-1}$. Set
\begin{equation}\label{EQ3CFF}
	\mathscr{C}:=\{f\in \sF: f|_{\fR_+}\in  \sC^+, f|_{\fR^3}\in \sC^3\}.
\end{equation}
It suffices to show $\sC$ is dense in $C_c(E)$ relative to the uniform norm and dense in $\sF$ relative to the $\sE_1$-norm. On one hand, $\sC$ is clearly an algebra, i.e. $f,g\in \sC$ implies $c_1\cdot f+c_2 \cdot g, f\cdot g\in \sC$ for any constants $c_1,c_2$. In addition, $\sC$ can separate points in $E$ by the following argument: Without loss of generality, consider $x\in \fR_+,y\in \fR^3\setminus \{\b0\}$. Since $\sC^+$ is a special standard core of $(\sE^+,\sF^+)$, there exists a function ${}^+f\in \sC^+$ such that ${}^+f(\b0)={}^+f(x)=1$. Another function ${}^3f\in \sC^3$ can be taken to separate $\b0$ and $y$, i.e. ${}^3f(\b0)\neq {}^3f(y)$. Define a function $f$ on $E$ by
\[
	f|_{\fR_+}:= {}^3f(\b0)\cdot {}^+f,\quad f|_{\fR^3}:={}^3f. 
\]
Then $f\in \sC$ and $f(x)={}^3f(\b0)\neq {}^3f(y)=f(y)$. Thus by the Stone-Weierstrass theorem, $\sC$ is dense in $C_c(E)$ relative to the uniform norm. On the other hand, fix $f\in \sF$ and a small constant $\varepsilon>0$. Take ${}^+g\in \sC^+$ with ${}^+g(\b0)=1$ and ${}^3g\in \sC^3$ with ${}^3g(\b0)=1$. Let $C_+:=\|{}^+g\|_{\sE^+_1}$ and $C_3:=\|{}^3g\|_{\sE^3_1}$. By \cite[Theorem~2.1.4]{FOT}, there exist two functions ${}^+h_\varepsilon \in \sC^+$ and ${}^3h_\varepsilon\in \sC^3$ such that 
\[
\begin{aligned}
&\|{}^+h_\varepsilon-f|_{\fR_+}\|_{\sE^+_1}<\frac{\varepsilon}{4\sqrt{\tp\vee 1}}, \quad |{}^+h_\varepsilon(\b0)-f(\b0)|<\frac{\varepsilon}{4C_+\sqrt{\tp\vee 1}};  \\
&\|{}^3h_\varepsilon-f|_{\fR^3}\|_{\sE^3_1}<\varepsilon/4, \quad |{}^3h_\varepsilon(\b0)-f(\b0)|<\frac{\varepsilon}{4C_3}.
\end{aligned}
\]
Define a function $f_\varepsilon$ on $E$ by
\[
\begin{aligned}
	& f_\varepsilon|_{\fR_+}:={}^+h_\varepsilon+ \left(f(\b0)-{}^+h_\varepsilon(\b0) \right)\cdot {}^+g, \\
	& f_\varepsilon|_{\fR^3}:={}^3h_\varepsilon+ \left(f(\b0)-{}^3h_\varepsilon(\b0) \right)\cdot {}^3g.
\end{aligned}
\]
Then $f_\varepsilon\in \sC$ and 
\[
\begin{aligned}
\|f_\varepsilon-f\|_{\sE_1}&\leq \sqrt{\tp\vee 1}\cdot\|f_\varepsilon|_{\fR_+}-f|_{\fR_+}\|_{\sE^+_1} +\|f_\varepsilon|_{\fR^3}-f|_{\fR^3}\|_{\sE^3_1} \\
&\leq \sqrt{\tp\vee 1}\cdot\|{}^+h_\varepsilon-f|_{\fR_+}\|_{\sE^+_1}+\sqrt{\tp\vee 1}\cdot|{}^+h_\varepsilon(\b0)-f(\b0)|\cdot \|{}^+g\|_{\sE^+_1} \\
 &\qquad \qquad\qquad  +\|{}^3h_\varepsilon-f|_{\fR^3}\|_{\sE^3_1}+|{}^3h_\varepsilon(\b0)-f(\b0)|\cdot \|{}^3g\|_{\sE^3_1} \\
 &<\varepsilon. 
\end{aligned}\]
This tells us that $\sC$ is dense in $\sF$ relative to the $\sE_1$-norm.

Furthermore, we derive the irreducibility of $(\sE,\sF)$. Take an $\fm$-invariant set $A\subset E$, and we need to show $\fm(A)=0$ or $\fm(A^c)=0$. Firstly, let $(\sA^+,\sG^+)$ be the part Dirichlet form of $(\sE,\sF)$ on $\fR_+\setminus \{\b0\}$ and consider $f,g\in \sG^+\subset \sF$. It follows from  \cite[Theorem~1.6.1]{FOT} that $f\cdot 1_A, g\cdot 1_A\in \sF$ and 
\[
	\sE(f,g)=\sE(f1_A,g1_A)+\sE(f1_{A^c},g1_{A^c}). 
\]
Set $A_+:=A\cap (\fR_+\setminus \{\b0\})$. Since $f|_{\fR^3}=g|_{\fR^3}\equiv 0$, the expression of $\sF$ yields $f\cdot 1_{A_+}, g\cdot 1_{A_+}\in \sG^+$ and
\[
	\sA^+(f,g)=\sA^+(f1_{A_+},g1_{A_+})+\sA^+(f1_{A_+^c},g1_{A_+^c}). 
\] 
Using \cite[Theorem~1.6.1]{FOT} again, we have $A_+$ is an $\fm_+$-invariant set relative to $\sA^+$. Note that $(\sA^+,\sG^+)$ is clearly irreducible and thus $\fm_+(A_+)=0$ or $\fm_+(A_+^c)=0$. Analogically set $A_3:=A\cap (\fR^3\setminus \{\b0\})$ and we can also obtain that $\fm_3(A_3)=0$ or $\fm_3(A_3^c)=0$. Secondly, it suffices to show that $\fm_+(A_+)=\fm_3(A_3^c)=0$ or $\fm_+(A_+^c)=\fm_3(A_3)=0$ is impossible. Take a function $f\in \sC$ such that $f(x)=1$ for $|x|\leq 1$. These two cases both contradict to $f\cdot 1_A\in \sF$. Eventually the irreducibility of $(\sE,\sF)$ is concluded.  


Finally, the part Dirichlet form of $(\sE,\sF)$ on $\fR_+\setminus \{\b0\}$ (resp. $\fR^3\setminus \{\b0\}$) is clearly associated with the same Markov process as that of $(\sE^+,\sF^+)$ (resp. $(\sE^3,\sF^3)$) on $\fR_+\setminus \{\b0\}$ (resp. $\fR^3\setminus \{\b0\}$). Therefore, the associated Markov process of $(\sE,\sF)$ is nothing but the dBMVD by definition. This completes the proof. \qed
\end{proof}

It is worth noting that $M$ as well as $(\sE,\sF)$ is always conservative as will be shown in Remark~\ref{RM12}.

\begin{remark}
A number of irreducible Markov processes on $E$ can be analogically obtained by using other processes on $\fR_+$ and $\fR^3$, relative to which the origin is of positive capacity. Every irreducible and symmetric diffusion on $\fR_+$ (reflecting at $\b0$) is such an example on $\fR_+$, see e.g. \cite{LY172}. An example of pure-jump process on $\fR^3$ appeared in a recent work \cite{LX19}. 
\end{remark}

\subsection{Basics of $(\sE,\sF)$ and $M$}

Denote the capacities relative to $\sE$, $\sE^+$ and $\sE^3$ by $\text{Cap}$, ${}^+\text{Cap}$ and ${}^3\text{Cap}$ respectively. The following proposition characterizes the sets of capacity zero relative to $\sE$.

\begin{proposition}\label{PRO1}
The set $A\subset E$ is of capacity zero relative to $\sE$, if and only if $A\subset \fR^3\setminus \{\b0\}$ and $A$ is of capacity zero relative to $\sE^3$. Particularly for any $x\in \fR_+$, $\text{Cap}(\{x\})>0$ but for any $x\in \fR^3\setminus \{\b0\}$, $\text{Cap}(\{x\})=0$.
\end{proposition}  
\begin{proof}
Let $(\sA^+,\sG^+)$ (resp. $(\sA^3,\sG^3)$) be the part Dirichlet form of $(\sE^+,\sF^+)$ (resp. $(\sE^3,\sF^3)$) on $\fR_+\setminus \{\b0\}$ (resp. $\fR^3\setminus \{\b0\}$). Then $(\sA^+,\sG^+)$ and $(\sA^3,\sG^3)$ are also the part Dirichlet forms of $(\sE,\sF)$ on $\fR_+\setminus \{\b0\}$ and $\fR^3\setminus \{\b0\}$ respectively.
Since every singleton of $\fR_+$ is of positive capacity relative to $\sE^+$, it is of positive capacity relative to $\sA^+$ as well as $\sE$ by applying \cite[Theorem~4.4.3]{FOT}. Hence any set of capacity zero relative to $\sE$ must be a subset of $\fR^3\setminus \{\b0\}$. By using \cite[Theorem~4.4.3]{FOT} again, a set $B\subset \fR^3\setminus \{\b0\}$ is of capacity zero relative to $\sE$, if and only if $B$ is of capacity zero relative to $\sA^3$ as well as $\sE^3$. This completes the proof. 
\qed \end{proof}

The behaviour of $M$ near the origin $\b0$ is crucial to the understanding of it. As indicated in Proposition~\ref{PRO1}, $\b0$ is of positive capacity relative to $\sE$. This leads to 
\begin{equation}\label{EQ3PXS}
	\mathbb{P}_x(\sigma_{\b0}<\infty)>0
\end{equation}
for $\sE$-q.e. $x\in E$, where $\sigma_{\b0}:=\inf\{t>0:M_t=\b0\}$ is the hitting time of $\{\b0\}$, by applying \cite[Theorem~4.7.1~(i)]{FOT}. Furthermore, the origin $\b0$ is called regular for a set $B\subset E$ with respect to $M$ if $\bP_\b0(\sigma_B=0)=1$ where $\sigma_B:=\inf\{t>0:M_t\in B\}$. Then we have the following.

\begin{corollary}
$\b0$ is regular for $\fR_+\setminus \{\b0\}$, $\{\b0\}$ and $\fR^3\setminus \{\b0\}$ with respect to $M$ respectively. 
\end{corollary}
\begin{proof}
If $\IP_\b0(\sigma_{\b0}=0)=0$, then $\{\b0\}$ is a thin set. Thin set is always semipolar and thus $\fm$-polar. This contradicts to Proposition~\ref{PRO1}. Hence $\IP_\b0(\sigma_{\b0}=0)=1$ by the $0$-$1$ law.

Let $B:=\fR_+\setminus \{\b0\}$ or $\fR^3\setminus \{\b0\}$. The set of all regular points for $B$ is denoted by $B^r$. Since $B$ is open, it is also finely open. Thus $B\subset B^r\subset \fR_+$ or $\fR^3$. If $\b0\notin B^r$,  then $B$ would be finely open and finely closed simultaneously. By \cite[Corollary~4.6.3]{FOT}, we would have $B$ is invariant. Hence the irreducibility of $(\sE,\sF)$ would imply $\fm(B)=0$. This is impossible. Eventually we can conclude that $\b0\in B^r$, i.e. $\b0$ is regular for $\fR_+\setminus \{\b0\}$ or $\fR^3\setminus \{\b0\}$. This completes the proof. 
\qed \end{proof}

\begin{remark}
At a heuristic level, this fact tells us that starting from $\b0$, $M$ enters $\fR_+\setminus \{\b0\}$, $\{\b0\}$, and $\fR^3\setminus \{\b0\}$ immediately. The behaviour of $M$ near $\b0$ is very similar to that of the excursions of one-dimensional Brownian motion. 
\end{remark}

Denote the extended Dirichlet spaces of $(\sE^+,\sF^+)$ and $(\sE^3,\sF^3)$ by $\sF^+_\re$ and $\sF^3_\re$ respectively. Note that $\sF^+_\re=\{f:f\circ \iota_+\in \cF^+_\re\}$ where  
\[
\cF^+_\re=\cF^+_\rho:=\{g: g\text{ is absolutely continuous on }\bR_+, \; \int_{\bR_+}g'(r)^2m_+(dr)<\infty\}
\]
when $X^+$ is recurrent, i.e. $1/\rho\notin L^1(\bR_+)$, and 
\[
\cF^+_\re=\{g\in \cF^+_\rho: \lim_{r\uparrow \infty} g(r)=0\}
\]
when $X^+$ is transient, i.e. $1/\rho\in L^1(\bR_+)$; see e.g. \cite[Theorem~2.2.11]{CF}. The expression of $\sF^3_\re$ is stated in \cite[Corollary~3.5]{FL17}. 
The extended Dirichlet space $\sF_\re$ of $(\sE,\sF)$ is given as follows. 

\begin{proposition}\label{PRO2}
It holds
\[
	\sF_\mathrm{e}=\left\{ f :f<\infty,\fm\text{-a.e.}, f|_{\fR^+}\in \sF^+_\mathrm{e}, f|_{\fR^3}\in \sF^3_\mathrm{e}, \widetilde{f|_{\fR_+}}(\b0)=\widetilde{f|_{\fR^3}}(\b0) \right\}.
\]
\end{proposition}
\begin{proof}
Denote the family on the right hand side by $\sG$. Take an arbitrary function $f\in \sF_\re$. Let $\sC$ given by \eqref{EQ3CFF} be a special standard core of $(\sE,\sF)$. By \cite[Theorem~2.1.7]{FOT}, there exist $f_n\in \sC$ which are $\sE$-Cauchy and converge to $f$, $\sE$-q.e. as $n\rightarrow \infty$. Then it follows from Proposition~\ref{PRO1} that ${}^+f_n:=f_n|_{\fR_+}$ (resp. ${}^3f_n:=f_n|_{\fR_3}$) are $\sE^+$-Cauchy (resp. $\sE^3$-Cauchy) and ${}^+f_n\rightarrow f|_{\fR^+}$, $\sE^+$-q.e. (resp. ${}^3f_n\rightarrow f|_{\fR^3}$, $\sE^3$-q.e.). Particularly, $f|_{\fR^+}\in \sF^+_\mathrm{e}, f|_{\fR^3}\in \sF^3_\mathrm{e}$ and $\widetilde{f|_{\fR_+}}(\b0)=\lim_{n\rightarrow \infty}{}^+f_n(\b0)=\lim_{n\rightarrow \infty}{}^3f_n(\b0)=\widetilde{f|_{\fR^3}}(\b0)$. This yields $\sF_\re\subset \sG$. 

To the contrary, take $f\in \sG$. Then ${}^+f:=f|_{\fR_+}$ admits an approximation sequence ${}^+f_n\in \sC^+$ with ${}^+f_n \rightarrow {}^+f$ pointwisely and ${}^3f:=f|_{\fR^3}$ admits an approximation sequence ${}^3f_n\in \sC^3$ with ${}^+f_n \rightarrow {}^3f$, $\sE^3$-q.e., where $\sC^+$ and $\sC^3$ are given in \eqref{EQ3CFF}. It follows that 
\[
	\lim_{n\rightarrow \infty}{}^+f_n(\b0)=\widetilde{{}^+f}(\b0)=\widetilde{{}^3f}(\b0)=\lim_{n\rightarrow \infty}{}^3f_n(\b0).
\]
In the case that a subsequence $\{n_k:k\geq 1\}$ of $\{n:n\geq 1\}$ exists such that ${}^+f_{n_k}(\b0)\neq 0$ for all $k$ (resp. ${}^3f_{n_k}(\b0)\neq 0$ for all $k$), the functions
\[
	f_k(x):=\left\lbrace 
	\begin{aligned}
	&\frac{{}^3f_{n_k}(\b0)}{{}^+f_{n_k}(\b0)}{}^+f_{n_k}(x),\quad x\in \fR_+, \\
	&{}^3f_{n_k}(x),\qquad\qquad\quad x\in \fR^3,
	\end{aligned} \right.\quad 
\left(\text{resp. }	f_k(x):=\left\lbrace 
	\begin{aligned}
	&{}^+f_{n_k}(x),\quad x\in \fR_+, \\
	&\frac{{}^+f_{n_k}(\b0)}{{}^3f_{n_k}(\b0)}{}^3f_{n_k}(x),\quad x\in \fR^3,
	\end{aligned} \right.\right)
\]
in $\sC$ constitute an approximation sequence of $f$. This leads to $f\in \sF_\re$. Otherwise we can assume without loss of generality that ${}^+f_n(\b0)={}^3f_n(\b0)=0$ for all $n$. Then $f_n:={}^+f_n$ on $\fR_+$ and $f_n:={}^3f_n$ on $\fR^3$ give an approximation sequence of $f$ and it follows $f\in \sF_\re$ as well. This completes the proof. 
\qed \end{proof}

The following corollary is a straightforward consequence of Proposition~\ref{PRO2}.  

\begin{corollary}
$(\sE,\sF)$ is recurrent if and only if both $(\sE^+,\sF^+)$ and $(\sE^3,\sF^3)$ are recurrent. 
Otherwise $(\sE,\sF)$ is transient.
\end{corollary}
\begin{proof}
It suffices to note that $(\sE,\sF)$ is recurrent, if and only if $1\in \sF_\re$ and $\sE(1,1)=0$.
\qed \end{proof}


\subsection{Generator of $(\sE,\sF)$}

The generator of $(\cE^+,\cF^+)$ is $\mathcal{A}^+:=\frac{1}{2}\frac{d^2}{d m_+ d \ts_+}$ with the domain  (see e.g. \cite{F14})
\[
 \cD(\cA^+):=\left\{f\in \cF^+: \frac{df}{d\ts_+}\ll m_+, \frac{d^2f}{dm_+ d\ts_+}\in L^2(\bR_+,m_+)  \right\}.
\]
Particularly, $C_c^\infty(\bR_+)\subset \cD(\cA^+)$. 
Then the generator of $(\sE^+,\sF^+)$ is $\sA^+:=\iota^*_+\cA^+$, where $(\iota^*_+ \cA^+) f:=\cA^+(f\circ \iota_+)$ for $f\in \cD(\sA^+):=\{g: g\circ \iota_+\in \cD(\cA^+)\}$. On the other hand, the generator $\cA^3:=\cA_\gamma$ of $(\cE^3,\cF^3)$ is given by \eqref{generator-3d-dBM}. Analogically the generator of $(\sE^3,\sF^3)$ is $\sA^3:=\iota^*_3 \cA^3$, where  $(\iota^*_3 \cA^3) f:=\cA^3(f\circ \iota_3)$ for $f\in \cD(\sA^3):=\{g: g\circ \iota_3\in \cD(\cA^+)\}$.

Denote the generator of $(\sE,\sF)$ by $\sA$ with the domain $\cD(\sA)$. Set $C_c^\infty(E):=\sC$ defined by \eqref{EQ3CFF} with $\sC^+:=C_c^\infty(\bR_+)\circ \iota_+^{-1}$ and $\sC^3:=C_c^\infty(\bR^3)\circ \iota_3^{-1}$. It is straightforward to verify that $C_c^\infty(E)\subset \cD(\sA)$ and for all $f\in C_c^\infty(E)$, 
\[
	\sA f|_{\fR_+}=\sA^+(f|_{\fR_+}),\quad \sA f|_{\fR^3}=\sA^3(f|_{\fR^3}).
\]
Define another operator on $L^2(E,\fm)$ as follows:
\[
\begin{aligned}
	&\cD(\sG):=\left\{f\in \sF: f|_{\fR_+}\in \cD(\sA^+), f|_{\fR^3}\in \cD(\sA^3) \right\}, \\
&\sG f|_{\fR_+}:=\sA^+(f|_{\fR_+}),\quad \sG f|_{\fR^3}:=\sA^3(f|_{\fR^3}),\quad \forall f\in \cD(\sG).  	
\end{aligned}\]
Clearly $C_c^\infty(E)\subset \cD(\sG)$ and $\sG|_{C_c^\infty(E)}=\sA|_{C_c^\infty(E)}$. Note that $\sG$ is not self-adjoint on $L^2(E,\fm)$, since $\cD(\sG)\subsetneqq \{f\in L^2(E,\fm): f|_{\fR_+}\in \cD(\sA^+),f|_{\fR^3}\in \cD(\sA^3)\}\subset \cD(\sG^*)$ where $\sG^*$ is the adjoint operator of $\sG$. Particularly, $\sA\neq \sG$. Furthermore, we have the following.

\begin{proposition}\label{prop3}
The following hold for $\sA$ and $\sG$:
\begin{itemize}
\item[(1)] $\sA$ is a self-adjoint extension of $\sG$ on $L^2(E,\fm)$.
\item[(2)] When $\rho+1/\rho\in L^1(\bR_+)$ and $\gamma>0$, $f\in \cD(\sG)$ if and only if $f\in \cD(\sA)$ and $\fm_+(\sA f|_{\fR_+})=\fm_3(\sA f|_{\fR^3})=0$. 
\end{itemize}
\end{proposition}
\begin{proof}
\begin{itemize}
\item[(1)] Take $f\in \cD(\sG)$. Then $f\in \sF$ and for any $g\in \sF$, 
\[
\begin{aligned}
	\sE(f,g)&=\tp\cdot\sE^+(f|_{\fR_+},g|_{\fR_+})+\sE^3(f|_{\fR^3},g|_{\fR^3})\\
	 &=(-\tp\cdot \sA^+(f|_{\fR_+}),g|_{\fR_+})_{L^2(\fR_+,\fm_+)}+(-\sA^3(f|_{\fR^3}),g|_{\fR^3})_{L^2(\fR^3,\fm_3)} \\
	 &=(-\sG f, g)_{L^2(E,\fm)},
\end{aligned}\]
where the second equality is due to $f|_{\fR_+}\in \cD(\sA^+)$ and $f|_{\fR^3}\in \cD(\sA^3)$. Hence we can conclude that $f\in \cD(\sA)$ and $\sA f=\sG f$. 
\item[(2)] When $\rho+1/\rho\in L^1(\bR_+)$ and $\gamma>0$, the constant functions belong to both $\sF^+$ and $\sF^3$. Take $f\in \cD(\sG)$. It follows from the first assertion that $\sA f=\sG f$, which yields $\sA f|_{\fR_+}=\sG f|_{\fR_+}=\sA^+(f|_{\fR_+})$. Thus $\fm_+(\sA f|_{\fR_+})=\fm_+\left(\sA^+(f|_{\fR_+})\right)=-\sE^+(f|_{\fR_+},1)=0$. Analogically we can obtain $\fm_3(\sA f|_{\fR^3})=0$. To the contrary, let $f\in \cD(\sA)$ such that $\fm_+(\sA f|_{\fR_+})=\fm_3(\sA f|_{\fR^3})=0$. Take arbitrary $g^+\in \sF^+$ and define a function $g$ on $E$ by letting $g|_{\fR_+}:=g^+$ and $g|_{\fR^3}\equiv g^+(\b0)$. 
Clearly $g\in \sF$ and it follows that
\[
\begin{aligned}
	\sE(f,g)&=(-\sA f, g)_{L^2(E,\fm)}\\
	&= \tp\cdot (-\sA f|_{\fR_+}, g^+)_{L^2(\fR_+,\fm_+)}+g^+(\b0)\cdot \fm_3(\sA f|_{\fR^3})\\
	&=\tp \cdot (-\sA f|_{\fR_+}, g^+)_{L^2(\fR_+,\fm_+)}. 
\end{aligned}\]
On the other hand, $\sE(f,g)=\tp\cdot\sE^+(f|_{\fR_+},g^+)+g^+(\b0)\cdot \sE^3(f|_{\fR^3},1)=\tp\cdot\sE^+(f|_{\fR_+},g^+)$. These yield for all $g^+\in \sF^+$,
\[
\sE^+(f|_{\fR_+},g^+)=\left(-\sA f|_{\fR_+}, g^+\right)_{L^2(\fR_+,\fm_+)}.
\]
Consequently, $f|_{\fR_+}\in \cD(\sA^+)$ and $\sA^+(f|_{\fR_+})= \sA f|_{\fR_+}$. Analogically we can obtain that $f|_{\fR^3}\in \cD(\sA^3)$ and $\sA^3(f|_{\fR^3})=\sA f|_{\fR^3}$. Eventually $f\in \cD(\sG)$. 
\end{itemize}This completes the proof. 
\qed \end{proof}

\begin{remark}
In the second assertion of Proposition~\ref{prop3}, $\fm_+(\sA f|_{\fR_+})=0$ or $\fm_3(\sA f|_{\fR^3})=0$ can be erased since for $f\in \cD(\sA)$, $\fm(\sA f)=-\sE(f,1)=0$ leads to $\tp\cdot \fm_+(\sA f|_{\fR_+})+\fm_3(\sA f|_{\fR^3})=0$.
\end{remark}

Let $L^2(E)$ be the $L^2$-space on $E$ endowed with the Lebesgue measures on $\fR_+$ and $\fR^3$ respectively. The notation $\mathbf{\Delta}$ denotes the Laplacian operator acting on $C_c^\infty(E\setminus \{\b0\})$, i.e. for any $f\in C_c^\infty(E\setminus \{\b0\})$ and $x=(x_1,x_2,x_3,r)\in E\setminus \{\b0\}$, $\mathbf{\Delta}f(x):=\sum_{i=1}^3\frac{\partial^2 f}{\partial x_i^2}(x)$ for $x\in \fR^3\setminus \{\b0\}$ and $\mathbf{\Delta}f(x):=\frac{\partial^2 f}{\partial r^2}(x)$ for $x\in \fR_+\setminus \{\b0\}$. Recall that $h_{\rho,\gamma}$ is defined by \eqref{EQ3HAG}. The following result is obvious by Lemma~\ref{LM1} and Example~\ref{EXA1}. 

\begin{proposition}
Consider $\tp=1$.
Take a constant $\alpha\in \bR$ and $\rho(r):=\re^{-2\alpha r}/\pi$. Set $h_{\alpha, \gamma}:=h_{\rho,\gamma}$. Then the following operator
\[
\begin{aligned}
&\cD(\sL_{\alpha, \gamma}):=\{f\in L^2(E): f/h_{\alpha,\gamma}\in \cD(\sA)\},\\
&\sL_{\alpha,\gamma}f:=h_{\alpha,\gamma}\cdot\sA\left(\frac{f}{h_{\alpha,\gamma}}\right)+\left(\frac{\alpha^2}{2}\cdot f|_{\fR_+}+\frac{\gamma^2}{2}f|_{\fR^3} \right),\quad f\in \cD(\sL_{\alpha,\gamma}),
\end{aligned}
\]
is a self-adjoint extension of $\mathbf{\Delta}$ (acting on $C_c^\infty(E\setminus \{\b0\})$) on $L^2(E)$. Furthermore, if $(\alpha_1,\gamma_1)\neq (\alpha_2,\gamma_2)$, then $\sL_{\alpha_1,\gamma_1}\neq \sL_{\alpha_2,\gamma_2}$. 
\end{proposition}

\subsection{Existence of transition density}

The following proposition states the existence of the transition density of $M$, which is the foundation of the study in \S\ref{Sec-short-time-HKE}. 

\begin{proposition}\label{PRO5}
The dBMVD $M$ satisfies the absolute continuity condition in the following sense: 
For any $x\in E$ and $t>0$, it holds $P_t(x,\cdot)\ll \fm$, where $\{P_t(x, \cdot)=\mathbb{P}_x(M_t\in \cdot): t\geq 0\}$ denotes the semigroup of $M$. Particularly there exists a density function $\{p(t,x,y): t>0, x,y\in E\}$ such that $P_t(x,dy)=p(t, x,y)\fm(dy)$.
\end{proposition}
\begin{proof}
By \cite[Theorem~4.2.4]{FOT}, it suffices to show that any $\fm$-polar set is polar (with respect to $M$). Let $B$ be such a nearly Borel $\fm$-polar set and set
\[
	\varphi(x):=\mathbb{E}_x \left(\mathrm{e}^{-\sigma_B};\sigma_B<\infty\right),\quad \forall x\in E,
\]
where $\sigma_B:=\inf\{t>0:M_t\in B\}$. Proposition~\ref{PRO1} indicates $B\subset \fR^3\setminus \{\b0\}$ and the definition of $\fm$-polar set tells us $\varphi=0$, $\fm$-a.e. We need to show $\varphi(x)=0$ for every $x\in E$ to conclude that $B$ is polar. Clearly $\varphi$ is q.e. finely continuous (see e.g. \cite[Theorem~4.2.5]{FOT}) and hence $\varphi(x)=0$ for q.e $x\in E$ by applying \cite[Lemma~4.1.5]{FOT}. Particularly $\varphi(x)=0$ for all $x\in \fR_+$ due to Proposition~\ref{PRO1}. Let $M^{3,\b0}$ be the part process of $M$ on $\fR^3\setminus \{\b0\}$. Note that $M^3$ satisfies the absolute continuity condition as mentioned in Remark~\ref{RM1}, and thus so does $M^{3,\b0}$. Fix $x\in \fR^3\setminus \{\b0\}$. Then
\[
\varphi(x)=\mathbb{E}_x\left(\mathrm{e}^{-\sigma_B}; \sigma_B<\sigma_{\b0} \right)+\mathbb{E}_x\left(\mathrm{e}^{-\sigma_B}; \sigma_B\geq \sigma_{\b0} \right).
\] 
Since $M^{3,\b0}$ satisfies the absolute continuity condition, it follows that $B$ is polar with respect to $M^{3,\b0}$ and $\mathbb{E}_x\left(\mathrm{e}^{-\sigma_B}; \sigma_B<\sigma_{\b0} \right)=0$. On $\{\sigma_B\geq \sigma_{\b0}\}$, $\sigma_B=\sigma_B\circ \theta_{\sigma_{\b0}}+\sigma_{\b0}$. By denoting the filtration of $M$ by $(\mathcal{F}_t)_{t\geq 0}$ and using strong Markovian property, we have
\[
\begin{aligned}
\mathbb{E}_x\left(\mathrm{e}^{-\sigma_B}; \sigma_B\geq \sigma_{\b0}\right)&=\mathbb{E}_x\left(\mathrm{e}^{-\sigma_B\circ \theta_{\sigma_{\b0}}}\cdot \mathrm{e}^{-\sigma_{\b0}};\sigma_B\geq \sigma_{\b0}\right)\\
&\leq \mathbb{E}_x\left(\mathrm{e}^{-\sigma_B}\circ \theta_{\sigma_{\b0}};\sigma_B\geq \sigma_{\b0}\right) \\
&\leq \mathbb{E}_x\left(\mathbb{E}_x\left(\mathrm{e}^{-\sigma_B}\circ \theta_{\sigma_{\b0}}|\mathcal{F}_{\sigma_{\b0}}\right)\right)\\
&=\mathbb{E}_x\left(\mathbb{E}_{M_{\sigma_{\b0}}}\left(\mathrm{e}^{-\sigma_B}\right) \right)=\mathbb{E}_x\left(\mathbb{E}_{\b0}\left(\mathrm{e}^{-\sigma_B}\right) \right) \\
&=\mathbb{E}_x\left(\varphi(\b0)\right)=0.
\end{aligned}\]
Consequently $\varphi(x)=0$, which eventually yields that $B$ is polar with respect to $M$. This completes the proof. 
\qed \end{proof}

\section{Signed radial process of dBMVD}\label{SEC4}

Let $(\sE,\sF)$ be in Theorem~\ref{THM2} and $M$ be its associated dBMVD. This section is devoted to obtaining the expression of the signed radial process induced by $M$. Define a map $u:E\mapsto \bR$ as follows:
\[
	u(x):=\left\lbrace
		\begin{aligned}
		|x|,\quad &x=(x_1,x_2,x_3,0)\in \fR^3,\\
		-r,\quad &x=(0,0,0,r)\in \fR_+,
		\end{aligned} \right.
\]
and let $Y_t:=u(M_t)$ for any $t\geq 0$. Then $Y:=(Y_t)_{t\geq 0}$ is the so-called \emph{signed radial process} of $M$.  
Set $\ell:=\fm\circ u^{-1}$, which is a fully supported Radon measure on $\bR$. In practise, one can easily obtain 
\begin{equation}\label{EQ4LRG}
	\ell(dr)=\frac{\re^{-2\gamma |r|}}{\pi}dr|_{(0,\infty)} + \tp\cdot \rho(-r)dr|_{(-\infty, 0)}.
\end{equation}
The following proposition characterizes the signed radial process in terms of Dirichlet forms.

\begin{proposition}\label{PRO6}
$Y=(Y_t)_{t\geq 0}$ is an $\ell$-symmetric diffusion process on $\bR$. It is associated with the regular Dirichlet form on $L^2(\bR, \ell)$: 
\begin{equation}\label{EQ4FYF}
\begin{aligned}
	\sF^Y&=\{f\in L^2(\bR, \ell): f'\in L^2(\bR, \ell)\}, \\
	\sE^Y(f,g)&=\frac{1}{2}\int_\bR f'(x)g'(x)\ell(dx),\quad f,g\in \sF^Y,
\end{aligned}\end{equation}
where $f'$ stands for the weak derivative of $f$ for all $f\in \sF^Y$. 
\end{proposition}
\begin{proof}
To prove that $Y$ is a Markov process, we appeal to \cite[Theorem~(13.5)]{S88}. It suffices to show that for any bounded $f\in \mathcal{E}^u(\bR)$, there exists $g\in \cE^u(\bR)$ such that
\begin{equation}\label{EQ4PTF}
	P_t(f\circ u)=g\circ u,
\end{equation}	
where $\cE^u(\bR)$ is the family of all universally measurable functions on $\bR$ and $P_t$ is the semigroup of $M$. By the rotational invariance of $M^3$, it is not hard to find that for any $x, y \in E$ with $u(x)=u(y)=:r$, 
\[
	\int_E f(u(\cdot))\mathbb{P}_x(M_t\in \cdot)=\int_E f(u(\cdot)) \mathbb{P}_y(M_t\in \cdot). 
\]
This implies $P_t(f\circ u)(x)=P_t(f\circ u)(y)$. Set $g(r):=P_t(f\circ u)(x)$, which is a well-defined function on $\bR$ since $u$ is surjective. The universal measurability of $g$ is derived as follows. Since $u$ is continuous, it follows that $f\circ u\in \cE^u(E)$ and thus $P_t(f\circ u)\in \cE^u(E)$. For any set $A\in \cB(\bR)$, let $B_+:=g^{-1}(A)\cap (0,\infty)$ and $B_-:=g^{-1}(A)\cap (-\infty, 0]$. We have
\[
	\iota_3(B_+\times S^2)\cup \iota_+(-B_-)= (P_t(f\circ u))^{-1}(A)\in \cE^u(E),
\] 
where $S^2:=\{x\in \bR^3: |x|=1\}$. 
Hence $\iota_3(B_+\times S^2)= (P_t(f\circ u))^{-1}(A)\cap (\fR^3\setminus \{\b0\})\in \cE^u(E)$ and $\iota_+(-B_-)=(P_t(f\circ u))^{-1}(A)\cap \fR_+\in \cE^u(E)$. This leads to $B_+, B_-\in \cE^u(\bR)$ by the continuity of $\iota_+, \iota_3$. Therefore, $g^{-1}(A)=B_+\cup B_-\in \cE^u(\bR)$. 
By applying \cite[Theorem~(13.5)]{S88}, we can conclude that $Y$ is a Markov process and its transition semigroup is 
\[
	P^Y_tf:=g,
\]
where $f,g$ are in \eqref{EQ4PTF}. Moreover, for any two functions $f_1,f_2$, we have
\[
\begin{aligned}
	(P^Y_t& f_1,f_2)_{\ell}\\
	&=((P^Y_tf_1)\circ u, f_2\circ u)_\fm=(P_t(f_1\circ u), f_2\circ u)_\fm=(f_1\circ u, P_t(f_2\circ u))_\fm=(f_1, P^Y_t f_2)_{\ell}.
\end{aligned}
\]
This leads to the symmetry of $Y$.	By means of Yosida approximation, we can easily obtain that $Y$ is associated with the Dirichlet form on $L^2(\bR,\ell)$:
	\[
		\begin{aligned}
			\sF^Y&=\{f: f\circ u\in \sF\}, \\
		\sE^Y(f,f)&=\sE(f\circ u, f\circ u),\quad f\in \sF^Y. 
		\end{aligned}
	\]
A simple computation gives the expression \eqref{EQ4FYF} of $(\sE^Y, \sF^Y)$. The regularity of $(\sE^Y,\sF^Y)$ is also clear by virtue of \cite[Corollary~3.11]{LY172}. This completes the proof.
\qed \end{proof}
\begin{remark}\label{RM12}
By the expression of $(\sE^Y, \sF^Y)$, one can easily figure out (see e.g. \cite{LY172}) that $Y$ is an irreducible diffusion on $\bR$, whose speed measure is $\ell$ and scale function $\ts^Y$ is as follows: For $r\geq 0$, 
\[
\ts^Y(r)=\left\lbrace 
\begin{aligned}
\frac{\pi \re^{2\gamma r}-\pi}{2\gamma},\quad &\text{when }\gamma\neq 0, \\
\pi r,\quad &\text{when }\gamma=0;
\end{aligned}
\right.
\]
and for $r<0$,
\[
\ts^Y(r)=-\int_r^0 \frac{1}{\tp\cdot\rho(-s)}ds. 
\]
Not surprisingly, $Y$ is recurrent if and only if $1/\rho\notin L^1(\bR_+)$ and $\gamma\geq 0$. Otherwise $Y$ is transient. A straightforward computation yields that the infinities $\pm \infty$ are not approachable in finite time (cf. \eqref{EQ2LRT}) and hence $Y$ is conservative. This leads to the conservativeness of $M$. 

On the other hand, the symmetrizing measures of $Y$ are unique up to a constant in the sense that if a non-trivial measure $\mu$ is a symmetric measure of $Y$ then $\mu=c\cdot \ell$ for some constant $c$ (see e.g. \cite{YZ10}). This yields that for different $\tp$, $Y$ is different and therefore so is $M$. 
\end{remark}

From now on we impose the following condition on $\rho$:
\begin{description}
\item[(ACP)] $\rho$ is absolutely continuous and $\rho(r)>0$ for all $r\in \bR_+$. 
\end{description}
Note that \textbf{(ACP)} indicates \eqref{EQ3RRR}. 
Denote the family of probability measures of $Y$ by $\{\mathbb{P}^Y_r: r\in \mathbb{R}\}$.  
We next prove that $Y$ is a semimartingale with the quadratic variation process $\< Y\>_t=t$ and figure out the associated SDE for $Y$.
The symmetric semimartingale local time of $Y$ at $0$ is denoted by $L_t^0(Y)$,
that is,
 \begin{equation*}
L_t^0(Y):=\lim_{\varepsilon \downarrow 0}\frac{1}{2\varepsilon}\int_0^t 1_{(-\varepsilon, \varepsilon )}(Y_s)d\langle Y\rangle_s
= \lim_{\varepsilon \downarrow 0}\frac{1}{2\varepsilon}\int_0^t 1_{(-\varepsilon, \varepsilon )}(Y_s)ds,\quad \forall t\geq 0.
\end{equation*}
The main result of this section is as follows, and in its proof the celebrated Fukushima's decomposition is employed. 

\begin{theorem}\label{THM6}
Assume that \textbf{(ACP)} holds. The signed radial process $Y$ is a semimartingale whose quadratic variation process is $\<Y\>_t=t$. Furthermore for any $r\in \mathbb{R}$, $Y=(Y_t)_{t\geq 0}$ under the probability measure $\mathbb{P}^Y_r$ is the unique solution to the following well-posed SDE (that is,  this SDE has weak solutions and the pathwise uniqueness holds for it.):
\begin{equation}\label{EQ4DYD}
\begin{aligned}
	&dY_t=dB_t+b(Y_t)dt+\frac{1-\pi\tp\rho(0)}{1+\pi\tp\rho(0)}\cdot dL^0_t(Y), \\
	&Y_0=r,
	\end{aligned}
\end{equation}
where $B:=(B_t)_{t\geq 0}$ is a standard Brownian motion, $b$ is defined by
\begin{equation}\label{EQ3BRG}
b(r):=\left\lbrace 
\begin{aligned}
	-\gamma,\qquad & r\geq 0,\\
	\frac{-\rho'(-r)}{2\rho(-r)},\qquad &r< 0,
\end{aligned}
\right.
\end{equation}
and $L^0(Y)=(L^0_t(Y))_{t\geq 0}$ is the symmetric semimartingale local time of $Y$ at $0$.
\end{theorem}
\begin{proof}
We first show that $Y$ is a semimartingale. Take $f(r):=r\in \sF^Y_\mathrm{loc}$ and consider the Fukushima's decomposition for $f$:
\[
	f(Y_t)-f(Y_0)=M^f_t+N^f_t. 
\]
The martingale part $M^f$ is determined by its energy measure $\mu_{\<f\>}$ and for any $g\in C_c^\infty(\bR)$ (see \cite[Theorem~5.5.2]{FOT}),
\[
	\int gd\mu_{\<f\>}= 2\sE^Y(fg,f)-\sE^Y(f^2,g)=\int gd\ell. 
\]
It follows that $\mu_{\<f\>}=\ell$ and hence $M^f$ is equivalent to a standard Brownian motion. For the zero-energy part $N^f$, we note
\[
	-\sE^Y(f,g)=-\frac{1}{2}\int_\bR g'(r)\ell(dr)=\frac{1-\pi\tp\rho(0)}{2\pi}\cdot g(0)- \int_{-\infty}^0 g(r)\frac{\rho'(-r)}{2\rho(r)}\ell(dr)- \gamma \int_0^\infty g(r)\ell(dr). 
\]
Thus \cite[Corollary~5.5.1]{FOT} yields that $N^f$ is of bounded variation, and its associated signed smooth measure is
\[
	\mu_{N^u}=\frac{1-\pi\tp\rho(0)}{2\pi}\cdot \delta_0+b(r)\ell(dr). 
\]
Eventually, we can conclude
\begin{equation}\label{EQ4YTY}
	Y_t-Y_0=B_t+\int_0^t b(Y_s)ds+\frac{1-\pi\tp\rho(0)}{2\pi}\cdot l^0_t,\quad t\geq 0, 
\end{equation}
where $(B_t)$ is a certain standard Brownian motion and $l^0:=(l^0_t)_{t\geq 0}$ is the the local time of $Y$ at $0$, i.e. is the positive continuous additive functional of $M$ having Revuz measure $\delta_0$. Particularly, $Y$ is a semimartingale and $\<Y\>_t=t$. 

Since $\rho(r)>0$ for all $r\geq 0$, it is straightforward to verify that $b\in L^1_\mathrm{loc}(\bR)$. The well-posedness of \eqref{EQ4DYD} is concluded by e.g. \cite[Theorem~7.1]{L18}. It suffices to note that \cite[Lemma~4.3]{L18} yields 
\[
	L^0_t(Y)=\frac{1+\pi\tp\rho(0)}{2\pi}\cdot l^0_t. 
\]
Therefore \eqref{EQ4YTY} implies that $Y$ is a weak solution to \eqref{EQ4DYD}. This completes the proof. 
\qed \end{proof}

The weight parameter $\tp$ appearing in the symmetric measure $\fm$ plays a role of so-called ``skew'' constant. When $\rho\equiv 1$ and $\gamma=0$, $Y$ is noting but the well-known skew Brownian motion with the skew constant $\frac{1}{1+\pi\tp}$, which behaves like a Brownian motion except for the sign of each excursion is chosen by using an independent Bernoulli random variable of the parameter $\frac{1}{1+\pi\tp}$. For general $\rho$ and $\gamma$, $Y$ is called a \emph{general skew Brownian motion} in a recent work \cite{L18}. The non-skew case means that the last term in \eqref{EQ4DYD} disappears, i.e. $\tp=\frac{1}{\pi\rho(0)}$, and clearly the following corollary holds. 

\begin{corollary}\label{COR5}
When $\tp=\frac{1}{\pi\rho(0)}$, $Y$ is the unique solution to the SDE
\[
	dY_t=dB_t+b(Y_t)dt,
\]
where $B=(B_t)_{t\geq 0}$ is a standard Brownian motion and $b$ is defined by \eqref{EQ3BRG}.
\end{corollary} 

We end this section with a remark for the condition \textbf{(ACP)}. Firstly, $\rho>0$ is only employed to conclude the well-posedness of \eqref{EQ4DYD}. In fact when $\rho(0)=0$ (for example $\rho(r):=|r|^\alpha$ for a constant $0<\alpha<1$), the other derivation still works and $Y$ is a weak solution to 
\begin{equation}\label{EQ4YTW}
	dY_t=dB_t+b(Y_t)dt+dL^0_t(Y).
\end{equation}
Note that $b$ is independent of $\tp$, and different $\tp$ brings different signed radial process as explained in Remark~\ref{RM12}. Hence \eqref{EQ4YTW} has infinite weak solutions. The point $0$ where $L^0(Y)$ locates is usually called a \emph{barrier} for \eqref{EQ4YTW}, as in the reduced case $b\equiv 0$ (though this could not happen in \eqref{EQ4DYD} because $b\equiv 0$ implies that $\rho$ is constant, which contradicts to $\rho(0)=0$), the solution to \eqref{EQ4YTW} is nothing but the reflecting Brownian motion on $\bR_+$. At this time, $Y$ runs on $\bR_+$ and cannot go across the barrier $0$ to reach the left axis. However the presence of $b$ in \eqref{EQ4YTW} leads to infinite solutions, which are all irreducible as mentioned in Remark~\ref{RM12}. That means $Y$ starting from everywhere can reach every point of $\bR$, and the barrier $0$ is definitely fake for it. In \cite[\S7.2]{L18}, this kind of barriers are called \emph{pseudo barriers} and we also refer more discussions about the equations \eqref{EQ4DYD} and \eqref{EQ4YTW} to \cite{L18}.
Secondly, the condition \textbf{(ACP)} can be weaken to
\begin{description}
\item[(BV)] $\rho$ is cadlag locally of bounded variation on $\bR_+$. 
\end{description}
Under \textbf{(BV)}, let $\nu_\rho$ be the signed Radon measure on $(0,\infty)$ induced by $\rho$ and set $\mu_\rho:=\hat{\mu}_\rho\circ u^{-1}$ on $(-\infty,0)$ with
\[
	\hat{\mu}_\rho(dr):=\frac{\nu_\rho(dr)}{\rho(r)+\rho(r-)},\quad r>0,
\] 
where $\rho(r-)$ is the left limit of $\rho$ at $r>0$. As stated in \cite[Lemma~5.2]{L18}, $Y$ is still a semimartingale with $\langle Y\rangle_t=t$ and a weak solution to the SDE:
\begin{equation}\label{EQ4YWG}
	dY_t=dB_t-\gamma 1_{(0,\infty)}(Y_t)dt+\frac{1-\pi\tp\rho(0)}{1+\pi\tp\rho(0)}\cdot dL^0_t(Y)+\int_{r\in(-\infty,0)}\mu_\rho(dr)dL^r_t(Y),
\end{equation}
where $(L^r_t(Y))_{t\geq 0}$ is the symmetric semimartingale local time of $Y$ at $r<0$. It is worth noting that \textbf{(BV)} is the weakest assumption for the derivation of \eqref{EQ4YWG}, whose well-posedness holds under the following assumption with the convention $\rho(0-):=\rho(0)$ (see e.g. \cite[Lemma~6.3]{L18}):
\begin{description}
\item[(P)] $\rho(r),\rho(r-)>0$ for all $r\in \bR_+$. 
\end{description}
Clearly \textbf{(ACP)} implies \textbf{(BV)} and \textbf{(P)}. If we denote the absolute continuous part and singular part of $\mu_\rho$ by $b_\rho(r)dr$ and $\kappa_\rho$ respectively, i.e. $\mu_\rho(dr)=b_\rho(r)dr+\kappa_\rho(dr)$, then the last term on the right hand side of \eqref{EQ4YWG} is equal to
\[
	\int_{r\in (-\infty,0)}b_\rho(r)drdL^r_t(Y)+\int_{r\in(-\infty,0)}\kappa_\rho(dr)dL^r_t(Y)=b_\rho(Y_t)dt+\int_{r\in(-\infty,0)}\kappa_\rho(dr)dL^r_t(Y)
\]
by applying the occupation times formula. The condition \textbf{(ACP)} indicates $\kappa_\rho=0$ and meanwhile \eqref{EQ4YWG} reduces to \eqref{EQ4DYD}.

\section{Short-time heat kernel estimate for dBMVDs}\label{Sec-short-time-HKE}

In this section, utilizing the SDE characterization for the radial process of $M$ derived in Theorem~\ref{THM6}, we establish the two-sided short-time heat kernel estimate for $M$, i.e., for $t\le T$ with an arbitrary $0<T<\infty$. As in Proposition~\ref{PRO5}, $p(t,x,y)$ denotes the transition density of $M$ with respect to $\mathfrak{m}$. Recall that $|\cdot|$ denotes the Euclidean distance on $\fR_+$ as well as on $\fR^3$, and  by slightly abusing the notation, 
\begin{equation}\label{EQ5XYX}
|x-y|:=|x-\b0|+|y-\b0|, \quad x\in \fR^3,\, y\in \fR_+.
\end{equation} 
Before introducing the main result of this section, we state the following definition for the Kato class $\mathbf{K}_{1,1}$. For the definition of general Kato class $\mathbf{K}_{n,d}$ with $n,d\in \bN$, see e.g. \cite{ChenZ}. 

\begin{definition}[Kato class $\mathbf{K}_{1, 1}$]\label{def-kato-class}
A function $f: \IR\rightarrow \IR$ is called in Kato class $\mathbf{K}_{1,1}$, denoted by $f\in \mathbf{K}_{1,1}$, if $\sup_{x\in \IR}\int_{|x-y|\le 1}\left|f(y)\right|dy<\infty$.
\end{definition}

Recall that Proposition~\ref{PRO5} states the existence of the transition density $p(t,x,y)$ of $M$ but in the sense of almost everywhere. 
The main result of this section below claims its continuity and obtains its short-time estimate. Note that every function $g$ defined on $\bR_+$ is regarded as the one on $\bR$ by imposing $g|_{(-\infty,0)}=0$ if no confusions caused. 

\begin{theorem}\label{main-thm}
Assume \eqref{EQ3DRR}, \textbf{(ACP)} and that
\begin{equation}\label{assumption-kato}
\frac{\rho'}{\rho}\in \{f:\, |f|^2\in \mathbf{K}_{1, 1}\}.
\end{equation}
Then the transition density $p(t,x,y)$ of the dBMVD $M$ with respect to $\fm$ is jointly continuous on $(0,\infty)\times E\times E$. Furthermore for any fixed $0<T<\infty$, there exist positive constants $C_i$, $1\le i\le 12$, depending on $\rho, \gamma, \tp, T$ such that $p(t,x,y)$  satisfies the following estimate: When $t\in (0, T]$,
\begin{description}
\item{\rm (i)} For $x,y \in \fR_+$, 
\begin{equation}\label{EQ5CTE1}
\frac{C_1}{\sqrt{t}\rho(|y|)} \re^{-\frac{C_2|x-y|^2}{t}} \le p (t,x,y)\le\frac{C_3}{\sqrt{t}\rho(|y|)}\re^{-\frac{C_4|x-y|^2}{t}};
\end{equation}
\item{\rm (ii)} For $x\in \fR_+$ and $y\in \fR^3$,
\begin{equation}\label{EQ5CTE2}
\frac{C_5}{\sqrt{t}}\re^{2\gamma|y|-\frac{C_{6}|x-y|^2}{t}} \le p(t,x,y)\le
\frac{C_{7}}{\sqrt{t}}\re^{2\gamma|y|-\frac{C_{8}|x-y|^2}{t}};
\end{equation}
\item{\rm (iii)} For $x,y\in \fR^3$,
\begin{equation}\label{EQ5CTE3}
\frac{ C_{9}}{\sqrt{t}}\re^{2\gamma |y|-\frac{C_{10}\left(|x|+|y|\right)^2}{t}}+q(t,x,y) \le p(t,x,y)
\le \frac{C_{11}}{\sqrt{t}}\re^{2\gamma |y|-\frac{C_{12}\left(|x|+|y|\right)^2}{t}}+q(t,x,y),
\end{equation}
where
\begin{equation}\label{EQ5QTX}
q(t,x,y)=\sqrt{\frac{2\pi}{t^3}}|x||y|\re^{-\frac{\gamma^2}{2}t+\gamma (|x|+|y|)-\frac{|x-y|^2}{2t}}, \quad t>0,\, x,y\in \fR^3,
\end{equation}
is the transition density of killed distorted Brownian motion $M^{3,\b0}$ (see the statement after Definition~\ref{definition-dBMVD}) with respect to $\fm_3$. 
\end{description}
\end{theorem}
\begin{remark}
We point out that $L^{q}(\IR)\subset \left\{f: |f|^2\in \mathbf{K}_{1, 1}\right\})$ for all $q\in (1, +\infty]$.
This yields that if $\rho'/\rho \in L^q(\bR_+)$ for some $q\in (1,+\infty]$, then \eqref{assumption-kato} holds. 
An example satisfying all assumptions in Theorem~\ref{main-thm} is given in Example~\ref{EXA1}, i.e. $\rho(r):=\re^{-2\alpha r}/\pi$ for a constant $\alpha\in \bR$. In this case, $\rho'/\rho\equiv -2\alpha\in L^\infty(\IR_+)$. 
\end{remark}

The proof will be divided into several steps. To accomplish it, we prepare a lemma concerning the short-time heat kernel estimate for the signed radial process $Y$.

\begin{lemma}\label{short-time-HKE-radial}
Assume that the same conditions as Theorem~\ref{main-thm} hold. Set a measure on $\bR$
\[
	\wh \ell(dr):= \frac{2}{1+\kappa}dr|_{(-\infty,0)}+\frac{2}{1-\kappa}dr|_{(0,\infty)}
\] 
with $\kappa:=\left(1-\pi\tp\rho(0)\right)/\left(1+\pi\tp\rho(0)\right)$. 
Then the signed radial process $Y$ has a jointly continuous transition density function
$\wh{p}^{Y}(t, r_1, r_2)$ with respect to $\wh \ell$, i.e. $\mathbb{P}^Y_{r_1}(Y_t\in dr_2)=\wh{p}^{Y}(t, r_1, r_2)\wh \ell(dr_2)$ for all $t> 0$ and $r_1,r_2\in \mathbb{R}$, and $\wh{p}^Y$ is jointly continuous on $(0,\infty)\times \bR\times \bR$.
Furthermore, for every $T> 0$, there exist constants  $C_i>0$, $13\le i \le 16$, such that the following estimate holds:
\begin{equation}\label{eq-short-time-HKE-radial}
\frac{C_{13}}{\sqrt{t}} \re^{-C_{14}|r_1-r_2|^2/t} \le \wh{p}^{Y}(t,r_1,r_2)\le\frac{C_{15}}{\sqrt{t}} \re^{-C_{16}|r_1-r_2|^2/t},
\quad 0<t\leq T,\, r_1,r_2\in \bR.
\end{equation}
\end{lemma}
\begin{proof}
The idea of the proof to the estimate \eqref{eq-short-time-HKE-radial} is refereed to, for instance, \cite[Theorem~A]{Z1}. 
Note that $-1<\kappa<1$. Let $Z$ be the skew Brownian motion
$$ dZ_t = dB_t + \kappa\cdot dL^0_t (Z) ,
$$
where $B_t$ is a certain one-dimensional standard Brownian motion and $L^0_t (Z)$ is the symmetric semimartingale local time of $Z$ at $0$. Clearly $Z$ is symmetric with respect to $\wh \ell$ (see e.g. \cite{H81}) and
the transition density function $p^Z(t, r_1, r_2)$ of $Z$ with respect to $\wh \ell$ is explicitly known as follows:
(see e.g. \cite[III.(1.16)]{RY}):
\begin{equation}\label{EQ5PZR}
\begin{aligned}
p^Z&(t,r_1,r_2) \\
&=\frac{1-\kappa}{2}\left[g_t(r_2-r_1)+\kappa g_t(r_2+r_1)\right]\mathbf{1}_{\{r_1>0, r_2>0\}} +\frac{1-\kappa^2}{2}g_t(r_2-r_1)\mathbf{1}_{\{r_1\geq 0, r_2\leq 0\}} 
\\
&+\frac{1+\kappa}{2}\left[g_t(r_2-r_1)-\kappa g_t(r_2+r_1)\right]\mathbf{1}_{\{r_1<0, r_2<0\}}+\frac{1-\kappa^2}{2}g_t(r_2-r_1)\mathbf{1}_{\{r_1\leq 0, r_2\geq 0\}},
\end{aligned}
\end{equation}
where $g_t(r)=\re^{-r^2/2t}/\sqrt{2\pi t}$. 
Note that $p^Z$ is jointly continuous on $(0,\infty)\times \bR\times \bR$ and smooth  at $r_1, r_2\neq 0$.
In addition, one can verify directly that for $t>0$ and $r_1\neq 0$,
\begin{equation}\label{EQ5RRN}
	\bR\ni r_2\mapsto \nabla_{r_1}p^Z(t,r_1,r_2)
\end{equation}
is continuous, and for some constants $c>0$ and $0<\alpha< \beta$ the following inequalities hold  for $t\in (0,T]$:
\begin{align}\label{z-two-sided}
p^Z(t, r_1, r_2)\leq c  t^{-1/2} \exp (-\alpha |r_1-r_2|^2/t),\quad r_1,r_2\in \bR,
\end{align}
and 
\begin{equation}\label{partial-z-two-sided}
| \nabla_{r_1} p^Z(t, r_1, r_2)| \leq c t^{-1} \exp (-\beta|r_1-r_2|^2/t),\quad r_1\neq 0,r_2\in \bR.
\end{equation}

The diffusion process $Y$ can be obtained from $Z$ through a drift perturbation
(i.e. Girsanov transform) induce by $b$ given by \eqref{EQ3BRG}. 
We now set $k_0(t,r_1,r_2)=p^Z(t,r_1,r_2)$, and then inductively define
\begin{equation}\label{EQ5KNT}
k_n(t,r_1,r_2):=\int_0^t \int_{\IR} k_{n-1}(t-s, r_1,r_3)\cdot b(r_3)\cdot \nabla_{r_3}p^Z(s,r_3,r_2)dr_3ds, \quad \text{for } n\ge 1.
\end{equation}
Before moving on, we record two estimates: Firstly, since $|b|^2\in \mathbf{K}_{1,1}$ due to \eqref{assumption-kato}, it holds
\begin{align*}
 \left(\int_{\IR}|b(r_3)|^2 \re^{-2(\beta-\alpha)|r_2-r_3|^2/T}\;dr_3\right)^{1/2}
&=\left(\sum_{i=0}^\infty\int_{i\le |r_3-r_2|\le i+1}|b(r_3)|^2 \re^{-2(\beta-\alpha)|r_2-r_3|^2/T}\;dr_3\right)^{1/2} \nonumber
\\
&\le \left(\sum_{i=0}^\infty \re^{-2(\beta-\alpha)i^2/T}\int_{i\le |r_3-r_2|\le i+1}|b(r_3)|^2 \;dr_3\right)^{1/2}\\& \stackrel{\eqref{assumption-kato}}{< } \infty. 
\end{align*}
Thus we can set that for some $0<c_1<\infty$ independent of $r_2$,
\begin{align}\label{e:42}
 \left(\int_{\IR}|b(r_3)|^2 \re^{-2(\beta-\alpha)|r_2-r_3|^2/T}\;dr_3\right)^{1/2}\leq c_1.
\end{align}
Secondly, it follows from \eqref{z-two-sided} and \eqref{partial-z-two-sided} that for $0<t\leq T$,
\[
\begin{aligned}
|k_1 (t,r_1,r_2)|&=\left|\int_0^t \int_{0\neq r_3\in \IR} k_{0}(t-s, r_1,r_3)\cdot b(r_3)\cdot \nabla_{r_3}p^Z(s,r_3,r_2)dr_3ds\right| 
\\
&\le  c^2 \int_0^t \int_{\IR}\frac{1}{\sqrt{t-s}}\re^{-\frac{\alpha |r_1-r_3|^2}{t-s}}\cdot |b(r_3)| \cdot\frac{1}{s}\re^{-\frac{\beta |r_3-r_2|^2}{s}}dr_3 ds 
\\
&=  c^2 \int_0^t \frac{1}{s^{3/4}}\frac{1}{(t-s)^{1/4}}\int_{\IR}\frac{1}{(t-s)^{1/4}}\re^{-\frac{\alpha |r_1-r_3|^2}{t-s}}|b(r_3)| \frac{1}{s^{1/4}}\re^{-\frac{\beta |r_3-r_2|^2}{s}}dr_3 ds
\\
&\le c^2 \int_0^t \frac{1}{s^{3/4}}\frac{1}{(t-s)^{1/4}} ds \cdot \left(\int_{\IR}\frac{1}{\sqrt{t-s}}\re^{-\frac{2\alpha |r_1-r_3|^2}{t-s}}\cdot\frac{1}{\sqrt{s}}\re^{-\frac{2\alpha |r_3-r_2|^2}{s}}dr_3 \right)^{1/2}
\\
& \qquad \times\left(\int_{\IR}|b(r_3)|^2 \re^{-\frac{2(\beta-\alpha)|r_2-r_3|^2}{T}}dr_3\right)^{1/2}.
\end{aligned}\]
A straightforward computation yields that $\int_0^t 1/\left(s^{3/4}(t-s)^{1/4}\right) ds$ is bounded by a constant independent of $t$. Then from the Chapman-Kolmogorov equation for Gaussian densities and \eqref{e:42} we have for some constant $c_2>0$,
\begin{equation}\label{eq:48}
\begin{aligned}
|k_1 & (t,r_1,r_2)|\leq c^2c_1c_2\frac{1}{t^{1/4}}\re^{-\frac{\alpha |r_1-r_2|^2}{t}}.
\end{aligned}\end{equation}
With these constants $c,c_1,c_2$ independent of $t,r_1,r_2$ at hand, set
\[
	t_0:=\left(\frac{1}{3cc_1c_2}\right)^4\wedge T, 
\]
and we will prove by induction that for all $n\in \bN$,
\begin{equation}\label{eq:47}
k_n(t,r_1,r_2)\leq \frac{c}{3^n}\cdot  t^{-1/2} \exp (-\alpha |r_1-r_2|^2/t),\quad 0<t\leq t_0, \, r_1,r_2\in \bR.
\end{equation}
Obviously \eqref{eq:47} holds for $n=0$ and assume that it holds for $n-1$. Similar to \eqref{eq:48}, we have for $0<t\leq t_0$,
\[
\begin{aligned}
|k_n(t,r_1,r_2)|&=\left|\int_0^t \int_{0\neq r_3\in \IR} k_{n-1}(t-s, r_1,r_3)\cdot b(r_3)\cdot \nabla_{r_3}p^Z(s,r_3,r_2)dr_3ds\right| 
\\
&\le  \frac{c^2}{3^{n-1}}  \cdot \int_0^t \int_{\IR}\frac{1}{\sqrt{t-s}}\re^{-\frac{\alpha |r_1-r_3|^2}{t-s}}\cdot |b(r_3)| \cdot\frac{1}{s}\re^{-\frac{\beta |r_3-r_2|^2}{s}}dr_3 ds 
\\
&\leq \frac{c^2c_1c_2}{3^{n-1}}\cdot t^{-1/4} \exp (-\alpha |r_1-r_2|^2/t) \\
&\leq \frac{c}{3^n}\cdot  t^{-1/2} \exp (-\alpha |r_1-r_2|^2/t). 
\end{aligned}
\]
Hence \eqref{eq:47} is concluded. 
It then follows that $\sum_{n=0}^\infty k_n(t, r_1, r_2)$ converges locally uniformly on $(t, r_1, r_2)\in (0, t_0]\times \bR\times \bR$. From here using the same argument as that in \cite[Lemma 3.17]{Lou}, one can further see that $\sum_{n=0}^\infty k_n(t, r_1, r_2)$  is absolutely convergent for $(t, r_1, r_2)\in (0, T]\times \IR \times \IR$ and indeed the transition density of $Y$ with respect to $\wh \ell$, i.e. 
\[
	\wh p^Y(t,r_1,r_2)=\sum_{n=0}^\infty k_n(t, r_1, r_2),\quad 0<t\leq T, \, r_1,r_2\in \bR.
\] 
Furthermore, it holds for some constant $c_3>0$ such that
\begin{equation}\label{EQ5PYT}
\wh{p}^Y(t,r_1, r_2)\le c_3 t^{-1/2} \exp\left(-\alpha |r_1-r_2|^2/t\right), \quad0<t\leq T,\, r_1,r_2\in \bR.
\end{equation}
By a standard chain argument (see e.g. \cite[pp. 36-37]{Lou}), it is not hard to see that the same Gaussian type lower bound holds. 

Finally let us prove the joint continuity of $\wh p^Y$ on $(0,\infty)\times \bR\times \bR$. We first show it for $t\in (0,t_0]$.  When $t\leq 0$, we write $k_n(t,r_1,r_2)=0$ for all $n\geq 0$ for convenience and \eqref{EQ5KNT} becomes
\begin{equation}\label{EQ5KNTR}
k_n(t,r_1,r_2):=\int_0^{t_0} \int_{\IR} k_{n-1}(t-s, r_1,r_3)\cdot b(r_3)\cdot \nabla_{r_3}p^Z(s,r_3,r_2)dr_3ds.
\end{equation}
For the sake of the local uniform convergence of $\sum_{n=0}^\infty k_n(t, r_1, r_2)$, it suffices to show the joint continuity of $k_n$. We utilize induction and clearly $k_0=p^Z$ is jointly continuous on $(0,t_0]\times \bR\times \bR$. Assume this holds for $k_{n-1}$. Take two arbitrary constants $R>0$ and $\delta<t_0/2$ and we turn to derive the joint continuity of $k_n$ at $(t^*,r^*_1,r^*_2)\in [\delta, t_0]\times [-R, R]\times [-R,R]$: Take an arbitrary sequence $(t_m, r^m_1,r^m_2)\rightarrow (t^*,r^*_1,r^*_2)$ on $[\delta, t_0]\times [-R,R]\times [-R,R]$ as $m\rightarrow \infty$ and we will verify
\begin{equation}\label{eq:51}
\lim_{m\rightarrow \infty} k_n(t_m,r^m_1,r^m_2)= k_n(t^*,r^*_1,r^*_2). 
\end{equation}
To do this, fix a small constant $\varepsilon<\delta/2$ and split the integrand on the right hand side of \eqref{EQ5KNTR} into three parts by multiplying $I_+(s):=1_{(0,\varepsilon]}(s)$, $I_\varepsilon(s):=1_{(\varepsilon, t-\varepsilon)}(s)$ and $I_-(s):=1_{[t-\varepsilon,t]}(s)$ respectively. Denote
\begin{equation}\label{EQ3JPV}
	J_{\pm,\varepsilon}(t,r_1,r_2):=\int_0^{t_0}\int_\bR I_{\pm,\varepsilon}(s) k_{n-1}(t-s, r_1,r_3)\cdot b(r_3)\cdot \nabla_{r_3}p^Z(s,r_3,r_2)dr_3ds.
\end{equation}
Then $k_n(t,r_1,r_2)=J_+(t,r_1,r_2)+J_\varepsilon(t,r_1,r_2)+J_-(t,r_1,r_2)$. 
Now we show 
\begin{equation}\label{eq:53}
\lim_{m\rightarrow \infty}J_{\pm,\varepsilon}(t_m,r^m_1,r^m_2)=J_{\pm,\varepsilon}(t^*,r^*_1,r^*_2).
\end{equation} 
by utilizing dominated convergence theorem respectively. For $J_\varepsilon
$, we derive the upper bound for the integrand in $J_\varepsilon$ firstly. Note that for $s\in (\varepsilon,t-\varepsilon)$, \eqref{partial-z-two-sided} and \eqref{eq:47} yield 
\[
|k_{n-1}(t-s,r_1,r_3)|\lesssim \frac{1}{\sqrt{\varepsilon}}
\]
and 
\[
\begin{aligned}
\left|\nabla_{r_3}p^Z(s,r_3,r_2)\right|&\lesssim \frac{1}{\varepsilon} \exp (-\beta|r_3-r_2|^2/{t_0}) \\
&\leq \frac{1}{\varepsilon}\left(1_{\{r:|r|\leq 2R\}}(r_3)+1_{\{r:|r|> 2R\}}(r_3)\cdot \exp (-\frac{\beta|r_3|^2}{4t_0})\right),
\end{aligned}\]
where the last inequality holds since for $|r_3|>2R>R\geq |r_2|$, $|r_3-r_2|\geq |r_3|/2$. Obviously 
\[
\frac{1}{\varepsilon^{3/2}}b(r_3)\left(1_{\{r:|r|\leq 2R\}}(r_3)+1_{\{r:|r|> 2R\}}(r_3)\cdot \exp (-\frac{\beta|r_3|^2}{4t_0})\right) \in L^1([0,{t_0}]\times \bR) 
\]
because of $|b|^2\in \mathbf{K}_{1,1}$. Due to the joint continuity of $k_{n-1}$ and \eqref{EQ5RRN}, it is easy to find that for a.e. $(s,r_3)$,
\[
\begin{aligned}
\lim_{m\rightarrow \infty}I_{\varepsilon}(s)& k_{n-1}(t_m-s, r^m_1,r_3)\cdot b(r_3)\cdot \nabla_{r_3}p^Z(s,r_3,r^m_2)\\
	&=I_{\varepsilon}(s) k_{n-1}(t^*-s, r^*_1,r_3)\cdot b(r_3)\cdot \nabla_{r_3}p^Z(s,r_3,r_2^*).
\end{aligned}\]  
Hence we can obtain \eqref{eq:53} for  $J_\varepsilon$.  To treat $J_-$, we utilize substitution as follows
\[
\begin{aligned}
J_-(t,r_1,r_2)&=
\int_{t-\varepsilon}^t \int_{\IR} k_{n-1}(t-s, r_1,r_3)\cdot b(r_3)\cdot \nabla_{r_3}p^Z(s,r_3,r_2)dr_3ds \\
&=\int_0^\varepsilon \int_{\IR} k_{n-1}(s, r_1,r_3)\cdot b(r_3)\cdot \nabla_{r_3}p^Z(t-s,r_3,r_2)dr_3ds.
\end{aligned}
\]
Analogically $|k_{n-1}(s,r_1,r_3)|\lesssim s^{-1/2}\in L^1([0,{t_0}])$ and  
\[
\begin{aligned}
\left|\nabla_{r_3}p^Z(t-s,r_3,r_2)\right|&\lesssim \frac{2}{\delta} \exp (-\beta|r_3-r_2|^2/{t_0}) \\
&\leq \frac{2}{\delta}\left(1_{\{r:|r|\leq 2R\}}(r_3)+1_{\{r:|r|> 2R\}}(r_3)\cdot \exp (-\frac{\beta|r_3|^2}{4t_0})\right).
\end{aligned}\]
Hence \eqref{eq:53} holds for $J_-$ by the same argument as that for $J_\varepsilon$. Finally, we use a generalized dominated convergence theorem \cite[\S2.3, Exercise~20]{F99} to establish \eqref{eq:53} for $J_+$, i.e.
\begin{equation}\label{eq-jointcont-plus}
\lim_{m\rightarrow \infty} J_+(t_m,r^m_1,r^m_2)= J_+(t^*,r^*_1,r^*_2). 
\end{equation}
For $s\leq \varepsilon<\delta/2$ and $\delta\leq t\leq t_0$, it follows from \eqref{eq:47} that 
\[
\begin{aligned}
	|k_{n-1}(t-s,r_1,r_3)|
	&\lesssim \sqrt{\frac{2}{\delta}}\exp\left(-\alpha |r_1-r_3|^2/t_0\right)\\
	&	 \leq \sqrt{\frac{2}{\delta}}\left(1_{\{r:|r|\leq 2R\}}(r_3)+1_{\{r:|r|> 2R\}}(r_3)\cdot \exp (-\frac{\alpha|r_3|^2}{4t_0})\right)=: g(r_3). 
\end{aligned}\]
Then for $s\leq \varepsilon$, 
\begin{equation}\label{EQ5KNTS}
	\left| k_{n-1}(t-s, r_1,r_3)\cdot b(r_3)\cdot \nabla_{r_3}p^Z(s,r_3,r_2) \right|\lesssim g(r_3)|b(r_3)|\cdot |\nabla_{r_3}p^Z(s,r_3,r_2)|.
\end{equation}
For notation convenience, we denote
\begin{equation*}
f_{t,r_1, r_2}(s, r_3):=k_{n-1}(t-s, r_1,r_3) b(r_3)\nabla_{r_3}p^Z(s,r_3,r_2),
\end{equation*}
and 
\[
h_{r_2}(s, r_3):=g(r_3) |b(r_3)|\cdot | \nabla_{r_3}p^Z(s,r_3,r_2)|. 
\]
Then \eqref{eq-jointcont-plus} amounts to
\begin{equation*}
\lim_{m\rightarrow \infty} \int_{\IR} \int_0^\eps f_{t_m,r^m_1, r^m_2}(s, r_3)ds dr_3 = \int_{\IR} \int_0^\eps  f_{t^*,r^*_1, r^*_2}(s, r_3)dsdr_3. 
\end{equation*}
Towards this by applying \cite[\S2.3, Exercise~20]{F99}, it suffices to verify the following conditions:
\begin{description}
\item[\rm (i)]$|f_{t_m, r^m_1, r^m_2}(s, r_3)|\le  h_{r^m_2}(s, r_3)$ for all $0<s\le \eps$, and a.e. $r_3\in \IR$;
\item[\rm (ii)] $\lim_{m\rightarrow \infty}f_{t_m,r^m_1, r^m_2}(s, r_3)=f_{t^*,r_1^*, r^*_2}(s, r_3)$ for all $0<s\le \eps$ and a.e. $r_3\in \bR$;
\item[\rm (iii)] $\lim_{m\rightarrow \infty}h_{r^m_2}(s, r_3)=h_{r^*_2}(s, r_3)$ for all $0<s\le \eps$ and a.e. $r_3\in \bR$;
\item[\rm (iv)] $\lim_{m\rightarrow \infty}  \int_{\IR} \int_0^\eps h_{r^m_2}(s, r_3)dsdr_3 = \int_{\IR} \int_0^\eps h_{r^*_2}(s, r_3)dsdr_3$.
\end{description}
Obviously (i) holds in view of \eqref{EQ5KNTS}. Since $k_{n-1}$ is jointly continuous and $r_2\mapsto \nabla_{r_3}p^Z(s,r_3,r_2)$ is continuous for $s>0$ and $r_3\neq 0$, (ii) and (iii) are also both true. To verify (iv), note that
\[
\int_{\IR} \int_0^\eps h_{r^m_2}(s, r_3)dsdr_3=\int_{\IR} g(r_3) |b(r_3)|\left( \int_0^\eps |\nabla_{r_3}p^Z(s,r_3,r^m_2)|ds\right)dr_3
\]
and $g(r_3)|b(r_3)|\in L^1(\bR)$.  Hence we only need to show that for a.e. $r_3\in \bR$, $\int_0^\eps |\nabla_{r_3}p^Z(s,r_3,r^m_2)|ds$ is bounded and
\begin{equation}\label{eq-54}
	\lim_{m\rightarrow \infty}\int_0^\eps |\nabla_{r_3}p^Z(s,r_3,r^m_2)|ds=\int_0^\eps |\nabla_{r_3}p^Z(s,r_3,r^*_2)|ds. 
\end{equation}
To accomplish these, it follows from \eqref{EQ5PZR} that for $r_3\neq 0, \pm r^m_2$,
\begin{equation}\label{EQ5NRP}
	|\nabla_{r_3}p^Z(s,r_3,r^m_2)|\le 4\left(\left|-\frac{r^m_2-r_3}{s}g_s(r^m_2-r_3)\right|+\left|\frac{r^m_2+r_3}{s}g_s(r^m_2+r_3)\right|\right).
\end{equation}
By computation we have
\begin{equation}\label{eq-56}
\left|r^m_2-r_3\right|\int_0^\varepsilon \frac{g_s(r^m_2-r_3)}{s}ds=\frac{1}{\sqrt{2\pi}}\int_{\frac{|r^m_2-r_3|^2}{\varepsilon}}^\infty \frac{\re^{-\tilde{s}/2}}{\sqrt{\tilde{s}}}d\tilde{s} \le 1,
\end{equation}
where for the ``$=$" we use the substitution $\tilde{s}:=|r^m_2-r_3|^2/s$. Similarly, 
\begin{equation}\label{eq-57}
\left|r^m_2+r_3\right|\int_0^\varepsilon \frac{g_s(r^m_2+r_3)}{s}ds=\frac{1}{\sqrt{2\pi}}\int_{\frac{|r^m_2+r_3|^2}{\varepsilon}}^\infty \frac{\re^{-\tilde{s}/2}}{\sqrt{\tilde{s}}}d\tilde{s} \le 1,
\end{equation}
with the substitution $\tilde{s}:=|r^m_2+r_3|^2/s$.
Plugging the above upper bounds  back into \eqref{EQ5NRP}, we get for all $r_3$ with $r_3\neq 0, \pm r^m_2$,
\[
\int_0^\varepsilon |\nabla_{r_3}p^Z(s,r_3,r^m_2)|ds\le 8. 
\]
For \eqref{eq-54}, we appeal to the generalized dominated convergence theorem \cite[\S2.3, Exercise~20]{F99} as well. Indeed, fix $r_3\neq 0, \pm r_2^m$ and denote the right hand side of \eqref{EQ5NRP} by $j_{r_2^m}(s)$. Then
\[
	|\nabla_{r_3}p^Z(s,r_3,r^m_2)|\le  j_{r^m_2}(s).
\] 
Clearly for all $0<s\leq \varepsilon$, $|\nabla_{r_3}p^Z(s,r_3,r^m_2)|$ converges to $|\nabla_{r_3}p^Z(s,r_3,r^*_2)|$ and $j_{r^m_2}(s)$ converges to $j_{r^*_2}(s)$ as $m\uparrow \infty$. 
In addition, \eqref{eq-56} and \eqref{eq-57} tell us that the integration of $j_{r^m_2}(s)$ from $0$ to $\varepsilon$ also converges to that of $j_{r^*_2}(s)$. Hence \eqref{eq-54} is a consequence of \cite[\S2.3, Exercise~20]{F99}. 
Now with all (i)-(iv) having been verified, \eqref{eq-jointcont-plus} as well as \eqref{eq:51} can be concluded. By letting $\delta\downarrow 0$ and $R\uparrow \infty$, we eventually conclude the joint continuity of $k_n$ on $(0,t_0]\times \bR\times \bR$. This leads to the joint continuity of $\wh p^Y$ on $(0,t_0]\times \bR\times \bR$. For $t_0\leq t\leq \frac{3}{2}t_0$, note that 
\begin{equation}\label{eq:60}
	\wh p^Y(t,r_1,r_2)=\int_\bR \wh p^Y(t_0/2,r_1,r_3)\wh p^Y(t-t_0/2,r_3,r_2)\wh \ell(dr_3)
\end{equation}
and $\wh p^Y$ is bounded by \eqref{EQ5PYT}. Hence the integrand in \eqref{eq:60} is bounded by
\[
 \frac{c_3}{\sqrt{t_0/2}}\re^{-2\alpha |r_1-r_3|^2/{t_0}}\cdot \frac{c_3}{\sqrt{t-t_0/2}}\lesssim \re^{2\alpha r_1^2/t_0} \cdot \re^{-\alpha r^2_3/t_0},
\]
since $1/\sqrt{t-t_0/2}\leq \sqrt{2/t_0}$ and $|r_1-r_3|^2\geq r^2_3/2-r^2_1$. 
Then the dominated convergence theorem indicates the joint continuity of $\wh p^Y$ for $t\in [t_0,\frac{3}{2}t_0]$. By repeating this argument, one can eventually conclude the joint continuity of $\wh p^Y$ on $(0,\infty)\times \bR \times \bR$. 
This completes the proof. 
\qed \end{proof}
\begin{remark}
Clearly, 
\[
	\tilde{p}^Y(t,r_1,r_2):=\frac{2}{1+\kappa}\wh p^Y(t,r_1,r_2)\cdot 1_{\{r_2<0\}}+\frac{2}{1-\kappa}\wh p^Y(t,r_1,r_2)\cdot 1_{\{r_2>0\}}
\]
is the transition density function of $Y$ with respect to the Lebesgue measure. Note that $\tilde{p}^Y$ is not continuous at $r_1=0$ or $r_2=0$, unless $\kappa=0$, i.e. $\pi\tp\rho(0)=1$. It is easy to figure out that $\tilde{p}^Y$ satisfies the same Gaussian type estimate as \eqref{eq-short-time-HKE-radial} for $r_1,r_2\neq 0$. 
\end{remark}

The following corollary shows the joint continuity of the transition density function of $Y$ with respect to its symmetric measure $\ell$. 

\begin{corollary}
Let $\ell$ be the symmetric measure \eqref{EQ4LRG} of $Y$. Then $Y$ has a jointly continuous transition density function
$p^{Y}(t, r_1, r_2)$ with respect to $\ell$, i.e. $\mathbb{P}^Y_{r_1}(Y_t\in dr_2)=p^{Y}(t, r_1, r_2)\ell(dr_2)$ for all $t> 0$ and $r_1,r_2\in \mathbb{R}$, and $p^Y$ is jointly continuous on $(0,\infty)\times \bR\times \bR$.
\end{corollary}
\begin{proof}
It suffices to note that
\begin{equation*}
	p^Y(t,r_1,r_2):=\frac{2}{(1+\kappa)\tp \rho(-r_2)}\wh p^Y(t,r_1,r_2)\cdot 1_{\{r_2<0\}}+\frac{2\pi \re^{2\gamma |r_2|}}{1-\kappa}\wh p^Y(t,r_1,r_2)\cdot 1_{\{r_2\geq 0\}},
\end{equation*}
which is continuous at $r_2=0$. This completes the proof.\qed
\end{proof}

Now we turn to the proof of Theorem~\ref{main-thm}. 	
Lemma~\ref{short-time-HKE-radial} yields the two-sided estimate on $\wh{p}(t, x, y)$ when $x, y\in \fR_+$ since $M_t=\iota_+(-Y_t)$ when $M_t\in \fR_+$, and hence \eqref{EQ5CTE1} can be concluded.

\begin{proof}[Proof of \eqref{EQ5CTE1}]
Fix $x,y\in \fR_+$. Take arbitrary $0\leq a<b$, we have 
\[
\begin{aligned}
\int_{y\in \fR_+,a\leq |y|\leq b}p(t,x,y)\fm(dy)&=\mathbb{P}_x(M_t\in \fR_+, a\leq |M_t|\leq b)=\mathbb{P}^Y_{-|x|}(a \leq -Y_t\leq b)\\
&=\frac{2}{1+\kappa}\int_{-b}^{-a}\wh p^Y(t,-|x|, r)dr=\frac{2}{1+\kappa}\int_{a}^{b}\wh p^Y(t,-|x|, -r)dr.
\end{aligned}\]
It follows from $p(t,x,y)=p(t,x,\iota_+(|y|))$ and $\fm(dy)|_{\fR_+}=\tp\cdot m_+\circ \iota_+^{-1}(dy)=\tp\rho(|y|)d|y|$ that
\[
	\int_{y\in \fR_+, a\leq |y|\leq b}p(t,x,y)\fm(dy)=\tp\int_{a\leq |y|\leq b}p(t,x,\iota_+(|y|))\rho(|y|)d|y|
\]
and thus 
\begin{equation}\label{eq-p-rr-11}
p(t,x,y)=\frac{2}{(1+\kappa)\tp\rho(|y|)}\wh p^Y(t,-|x|,-|y|).
\end{equation}
Since $\left| |x|-|y| \right|=|x-y|$, 
Lemma~\ref{short-time-HKE-radial} immediately yields that
\begin{equation*}
\frac{c_1}{\sqrt{t}} \re^{-{c_2|x-y|^2}/{t}} \le \wh{p}^Y(t,-|x|,-|y|)\le\frac{c_3}{\sqrt{t}} \re^{-{c_4|x-y|^2}/{t}},
\qquad  t\in (0, T] \hbox{ and } x, y \in \fR_+.
\end{equation*}
Therefore the desired result \eqref{EQ5CTE1} follows from \eqref{eq-p-rr-11}. 
\qed \end{proof}



To prove \eqref{EQ5CTE2}, the crucial fact is that starting from $x\in \fR_+$ (resp. $y\in \fR^3$), $M$ must pass through the origin $\b0$ before reaching $y\in \fR^3$ (resp. $x\in \fR_+$). As usual $\sigma_{\b0}:=\inf\{t>0: M_t=\b0\}$ denotes the first hitting time of $\{\b0\}$ relative to $M$.  

\begin{proof}[Proof of \eqref{EQ5CTE2}]
Consider $x\in \fR_+$ and $y\in \fR^3$. We first note that in this case by the symmetry of $p(t, x, y)$ in $x$ and $y$,
\begin{equation*}
p(t,x,y) = p(t, y, x) =\int_{0}^t \IP_y(\sigma_{\b0}\in ds)p(t-s, \b0,x).
\end{equation*}
By the rotational invariance of three-dimensional distorted Brownian motion $X^3$, $\IP_y(\sigma_{\b0}\in ds)$ only depends
on $|y|$, therefore so does $y\mapsto p(t,x,y)$.
 For all $0\leq a<b$ and  $x\in \fR_+$, using polar coordinates we have
\begin{equation}\label{EQ5ABP}
\begin{aligned}
\frac{2}{1-\kappa}\int_a^b \wh{p}^{Y}(t,-|x|,r)dr
	&= \IP_{-|x|} ( a\leq Y_t\leq b) \\&= \IP_x ( M_t \in \fR^3 \hbox{ with }  a\leq |M_t | \leq b) 
\\&=  \int_{y\in {\fR^3}:  a \leq |y| \leq b}p(t,x,y)\fm(dy). 
\end{aligned}\end{equation}
Note that $\fm(dy)|_{\fR^3}=\left(h_{\rho,\gamma}(y)^2|y|^2d|y|d\sigma\right)\circ \iota_3^{-1}$, where $\sigma$ is the surface measure on the sphere $S^2$, and the density function $h_{\rho,\gamma}(y)^2=\frac{\re^{-2\gamma |y|}}{4\pi^2|y|^2}$ only depends on $|y|$ as well. Thus the last term in \eqref{EQ5ABP} is equal to
\[
	\int_{a\leq |y| \leq b}p(t,x,y)h_{\rho,\gamma}(y)^2|y|^2d|y|\int_{S^2}d\sigma=\int_{a\leq |y| \leq b}4\pi\cdot p(t,x,y)h_{\rho,\gamma}(y)^2|y|^2d|y|.
\]
 This yields 
\begin{equation}\label{e:retionship-radial-original}
\frac{2}{1-\kappa}\wh{p}^{Y}(t,-|x|,|y|) =
4\pi |y|^2p(t,x,y)h_{\rho,\gamma}(y)^2.
\end{equation}
By Lemma~\ref{short-time-HKE-radial} we can obtain that
\begin{equation}\label{EQ5CTE}
\frac{c_1}{\sqrt{t}}\re^{-c_2(|x|+|y|)^2/t} \le \wh{p}^Y(t,-|x|,|y|)\le
\frac{c_3}{\sqrt{t}}\re^{-c_4(|x|+|y|)^2/t}.
\end{equation}
In view of \eqref{EQ5XYX}, $|x-y|=|x|+|y|$ since $x\in \fR_+$ and $y\in \fR^3$. Eventually \eqref{EQ5CTE2} can be concluded from \eqref{e:retionship-radial-original} and \eqref{EQ5CTE}.
\qed
\end{proof}




Next we study the case that both $x$ and $y$ are in $\fR^3$. To continue, we first establish the explicit transition density function \eqref{EQ5QTX} for three-dimensional distorted Brownian motion $M^{3,\b0}$ killed upon hitting $\b0$.  Denote this transition density function by $q(t,x,y)$. In other words, for any non-negative function $f\geq 0$ on $\fR^3\setminus \{\b0\}$, 
\begin{equation*}
 \int_{\fR^3\setminus \{\b0\}}  q (t, x, y)  f(y) \mathfrak{m}(dy) = \IE_x \left( f(M_t); t< \sigma_{\b0} \right).
\end{equation*}
For $x=\b0$ or $y=\b0$, we make the convention $q(t,x,y):=0$.  The following result is the key ingredient of \eqref{EQ5CTE3}.

\begin{lemma}\label{HK-killed-dBM}
It holds that for $x,y\in \fR^3$ and $t>0$,
\begin{equation*}
q(t,x,y)=\re^{-\frac{\gamma^2}{2}t}\cdot \frac{\re^{-\frac{|x-y|^2}{2t}}}{(2\pi t)^{\frac{3}{2}}}\cdot \frac{1}{h_{\rho,\gamma} (x)h_{\rho,\gamma} (y)}. 
\end{equation*}
\end{lemma}
\begin{proof}
Theorem~\ref{THM3} tells us that $M^{3,\b0}$ is identified with $\iota_3({}_hW^\gamma)$, and it suffices to note that the transition semigroup of ${}_hW^\gamma$ is defined by \eqref{EQ5HPG}. \qed
\end{proof}

Now we have a position to complete the proof of \eqref{EQ5CTE3}. 

\begin{proof}[Proof of \eqref{EQ5CTE3}]
Consider $x,y\in \fR^3$. Note that starting from $x$, before hitting $\b0$, $M$ has the same distributions as those for $\iota_3(X^3)$, where $X^3$ is the three-dimensional dBM appearing in \S\ref{SEC2-1}.
Thus  $q(t,x,y)$ gives the probability density
that $M$ starting from $x$ hits $y$ at time $t$ without hitting $\b0$. 
As a consequence, 
\begin{equation}\label{decomposition-p}
\bar{q}(t,x,y):=p(t,x,y)-q(t,x,y),\quad x,y\in \fR^3 
\end{equation}
is the probability density for $M$ starting from $x$ hits $\b0$ before ending up at $y$ at time $t$, and this yields
\begin{equation*}
\bar{q}(t,x,y) =
\int_{0}^t p(t-s, \b0,y) \IP_x (\sigma_{\b0}\in ds).
\end{equation*}
As mentioned in the proof of \eqref{EQ5CTE2}, $ p(t-s, \b0,y)$
is  a function in $y$ depending only on $|y|$.  Therefore so is $y\mapsto
\bar{q}(t,x,y)$. Since $\bar q(t,x,y)=\bar{q}(t,y,x)$, $x\mapsto \bar{q}(t,x,y)$ also depends only on $|x|$. 
For any $b>a\geq 0$, it follows that
\[
\begin{aligned}
\IP_x \left( \sigma_{\b0}<t, \, M_t \in \fR^3 \hbox{ with } a\leq |M_t| \leq b\right)
&=\int_{a\le |y|\le b}\bar{q} (t,x,y)\mathfrak{m}(dy)
\\&=4\pi\int_{a\le |y|\le b}\bar{q} (t,x,y)h_{\rho,\gamma}(y)^2 |y|^2 d|y|.
\end{aligned}\]
On the other hand, $\IP_x \left( \sigma_{\b0}<t, \, M_t \in \fR^3 \hbox{ with } a\leq |M_t| \leq b\right)$ is also equal to
\[
\IP_{|x|}^{Y} \left( \sigma_{\b0}<t, a\leq |Y_t| \leq b\right) 
= \frac{2}{1-\kappa}\int_0^t \left(\int_a^b \wh{p}^{Y}(t-s, 0, r) dr \right) \IP_{|x|}^{Y}
\left(\sigma_{\b0} \in ds\right) .
\]
This yields $4\pi |y|^2 \bar{q}(t,x,y)h_{\rho,\gamma}(y)^2= \frac{2}{1-\kappa}\int_0^t \wh{p}^{Y}(t-s, 0, |y|)
\IP_{|x|}^{Y} \left(\sigma_{\b0} \in ds\right)$ and 
it follows from Lemma~\ref{short-time-HKE-radial} that
\[
\begin{aligned}
4\pi |y|^2 \bar{q}(t,x,y)h_{\rho,\gamma}(y)^2 & \leq \int_0^t \frac{\bar{c}_1}{\sqrt{t-s}}\re^{-\bar c_2|y|^2/(t-s)}
\IP_{|x| }^{Y} \left(\sigma_{\b0} \in ds\right)
\\
&\le \bar c_3 \int_0^t \wh{p}^{Y}(t-s, 0, -\bar c_4|y|)
\IP_{|x|}^{Y} \left(\sigma_{\b0} \in ds\right)
\\
&=\bar c_3\; \wh{p}^{Y}(t, |x|, - \bar c_4 |y|)  
\\
 &\le \frac{\bar c_5}{\sqrt{t}}\re^{- \bar c_6 \left(|x|+|y|\right)^2/t}.
\end{aligned}\]
By a piece of similar argument, one can  show that 
\begin{equation*}
4\pi |y|^2 \bar{q}(t,x,y)h_{\rho,\gamma}(y)^2  \ge \frac{ \bar c_7}{\sqrt{t}}\re^{- \bar c_8\left(|x|+|y|\right)^2/t}.
\end{equation*}
In other words, we have
\begin{equation*}
\frac{\bar c_7}{4\pi |y|^2\sqrt{t}}\re^{- \bar c_8\left(|x|+|y|\right)^2/t}
\leq  \bar{q}(t,x,y)h_{\rho,\gamma}(y)^2 \leq \frac{\bar c_5}{4\pi |y|^2\sqrt{t}}\re^{- \bar c_6 \left(|x|+|y|\right)^2/t}.
\end{equation*}
Since $h_{\rho,\gamma}(y)^2=\frac{\re^{-2\gamma |y|}}{4\pi^2|y|^2}$, this yields
\begin{equation}\label{HKE-q-bar}
\frac{\pi\bar c_7}{\sqrt{t}}\re^{2\gamma |y|-[ \bar c_8\left(|x|+|y|\right)^2/t]}
\leq  \bar{q}(t, x, y) \leq \frac{\pi\bar c_5}{\sqrt{t}}\re^{2\gamma |y|-[ \bar c_6 \left(|x|+|y|\right)^2/t]}.
\end{equation}
Combining \eqref{HKE-q-bar} with Lemma~\ref{HK-killed-dBM}, also in view of \eqref{decomposition-p}, we get for $x,y\in \fR^3$ that 
\begin{align*}
\nonumber &\frac{\pi\bar c_7}{\sqrt{t}}\re^{2\gamma |y|-[ \bar c_8\left(|x|+|y|\right)^2/t]}+q(t,x,y) \le p(t,x,y)
\le\frac{\pi\bar c_5}{\sqrt{t}}\re^{2\gamma |y|-[ \bar c_6 \left(|x|+|y|\right)^2/t]}+q(t,x,y). 
\end{align*}
 This completes the proof of \eqref{EQ5CTE3}.
\qed \end{proof}

Finally we prove the joint continuity of $p(t,x,y)$. 

\begin{proof}[Proof of joint continuity]
Clearly, \eqref{eq-p-rr-11} and \eqref{e:retionship-radial-original} tell us
\begin{equation}\label{eq-expression-p-12}
p(t,x,y)=\left\lbrace 
\begin{aligned}
&\frac{2}{(1+\kappa)\tp\rho(|y|)}\wh p^Y(t,-|x|,-|y|), \quad x,y\in \fR_+, \\
&\frac{2\pi}{1-\kappa}\re^{2\gamma|y|}\wh{p}^{Y}(t,-|x|,|y|), \quad \qquad x\in \fR_+,y\in \fR^3.
\end{aligned}
\right.
\end{equation}
Since $p(t,x,y)=p(t,y,x)$, it is straightforward to verify the joint continuity for $x\in \fR_+,y\in E$ and $x\in E, y\in \fR^3$. Now consider the case $x,y\in \fR^3$. Note that $q$ is jointly continuous by its explicit expression \eqref{EQ5QTX}. It suffices to obtain the joint continuity of $\bar{q}(t,x,y)=p(t,x,y)-q(t,x,y)$. To accomplish this, take $0\leq a<b$ and we have
\[
	\bP_x(M_t\in \fR^3, a\leq |M_t|\leq b)=\int_{a\leq |y|\leq b}q(t,x,y)\fm(dy)+\int_{a\leq |y|\leq b}\bar{q}(t,x,y)\fm(dy).
\]
The left hand side is also equal to
\[
\bP^Y_{|x|}(a\leq Y_t\leq b)=\frac{2}{1-\kappa}\int_a^b \wh p^Y(t,|x|,r)dr.
\]
A straightforward computation yields 
\[
\int_{a\leq |y|\leq b}q(t,x,y)\fm(dy)=\frac{\re^{-\frac{\gamma^2t}{2}}}{(2\pi t)^{3/2}}|x|\re^{\gamma|x|-\frac{|x|^2}{2t}}\int_a^b |y|\re^{-\gamma |y|-\frac{|y|^2}{2t}} \chi\left(\frac{|x||y|}{t}\right) d|y|,
\]
where $\chi(a):=\int_0^\pi \re^{a\cos \theta}\sin \theta d\theta$ is clearly a continuous function in $a\in \bR$. Hence we can obtain for $x,y\in \fR^3$,
\begin{equation}\label{eq-bar-q}
\bar{q}(t,x,y)=\frac{2\pi}{1-\kappa}\re^{2\gamma|y|}\wh p^Y(t,|x|,|y|)-\frac{\pi\re^{-\frac{\gamma^2t}{2}}}{(2\pi t)^{3/2}}|x|\re^{\gamma|x|-\frac{|x|^2}{2t}} |y|\re^{\gamma |y|-\frac{|y|^2}{2t}} \chi\left(\frac{|x||y|}{t}\right). 
\end{equation}
The joint continuity of $\bar{q}$ is obvious by this explicit expression. This completes the proof. 
\end{proof}
\begin{remark}
Note that $\bar{q}(t,x,y)$ in \eqref{eq-bar-q} depends only on $|x|$ and $|y|$. 
It is worth pointing out that for $x,y\in \fR^3$, $p(t,x,y)$ does not depend on $|x|$ or $|y|$ only, since  neither does $q(t,x,y)$. 
\end{remark}

\bibliographystyle{abbrv}
\bibliography{DBMV4}

\begin{thebibliography}{10}

\bibitem{AGH05}
S.~Albeverio, F.~Gesztesy, R.~H{\o}egh-Krohn, and H.~Holden.
\newblock {\em {Solvable models in quantum mechanics}}.
\newblock AMS Chelsea Publishing, Providence, RI, second edition, 2005.

\bibitem{AHS77}
S.~Albeverio, R.~H{\o}egh-Krohn, and L.~Streit.
\newblock {Energy forms, Hamiltonians, and distorted Brownian paths}.
\newblock {\em J. Math. Phys.}, 18(5):907--917, 1977.

\bibitem{CF}
Z.-Q. Chen and M.~Fukushima.
\newblock {\em {Symmetric Markov processes, time change, and boundary theory}}.
\newblock Princeton University Press, Princeton, NJ, 2012.

\bibitem{CF15}
Z.-Q. Chen and M.~Fukushima.
\newblock {One-point reflection}.
\newblock {\em Stochastic Process. Appl.}, 125(4):1368--1393, 2015.

\bibitem{CL}
Z.-Q. Chen and S.~Lou.
\newblock {Brownian motion on some spaces with varying dimension}.
\newblock {\em Ann. Probab.}, 47(1):213--269, 2019.

\bibitem{ChenZ}
Z.-Q. Chen and Z.~Zhao.
\newblock {Diffusion processes and second order elliptic operators with
  singular coefficients for lower order terms}.
\newblock {\em Math. Ann.}, 302:323--357, 1995.

\bibitem{CKMV2}
M.~Cranston, L.~Koralov, S.~Molchanov, and B.~Vainberg.
\newblock {Continuous model for homopolymers}.
\newblock {\em J. Funct. Anal.}, 256(8):2656--2696, 2009.

\bibitem{CKMV1}
M.~Cranston, L.~Koralov, S.~Molchanov, and B.~Vainberg.
\newblock {A solvable model for homopolymers and self-similarity near the
  critical point}.
\newblock {\em Random Oper. Stochastic Equations}, 18(1):73--95, 2010.

\bibitem{CM}
M.~Cranston and S.~Molchanov.
\newblock {On the critical behavior of a homopolymer model}.
\newblock {\em Sci. China Math.}, 62(8):1463--1476, 2019.

\bibitem{E90}
K.~B. Erickson.
\newblock {Continuous extensions of skew product diffusions}.
\newblock {\em Probab. Theory Related Fields}, 85(1):73--89, 1990.

\bibitem{FL17}
P.~J. Fitzsimmons and L.~Li.
\newblock {On the Dirichlet form of three-dimensional Brownian motion
  conditioned to hit the origin}.
\newblock {\em Sci. China Math.}, 62(8):1477--1492, 2019.

\bibitem{F99}
G.~B. Folland.
\newblock {\em {Real analysis}}.
\newblock Pure and Applied Mathematics (New York). John Wiley {\&} Sons, Inc.,
  New York, second edition, 1999.

\bibitem{F14}
M.~Fukushima.
\newblock {On general boundary conditions for one-dimensional diffusions with
  symmetry}.
\newblock {\em J. Math. Soc. Japan}, 66(1):289--316, 2014.

\bibitem{FO89}
M.~Fukushima and Y.~Oshima.
\newblock {On the skew product of symmetric diffusion processes}.
\newblock {\em Forum Math.}, 1(2):103--142, 1989.

\bibitem{FOT}
M.~Fukushima, Y.~Oshima, and M.~Takeda.
\newblock {\em {Dirichlet forms and symmetric Markov processes}}.
\newblock Walter de Gruyter {\&} Co., Berlin, 2011.

\bibitem{H81}
J.~M. Harrison and L.~A. Shepp.
\newblock {On skew Brownian motion}.
\newblock {\em Ann. Probab.}, 9(2):309--313, 1981.

\bibitem{K94}
T.~Kilpel{\"a}inen.
\newblock {Weighted Sobolev Spaces and Capacity}.
\newblock {\em Ann. Acad. Sci. Fenn. Ser. A I Math.}, 19(1):95--113, 1994.

\bibitem{L18}
L.~Li.
\newblock {On general skew Brownian motions}.
\newblock {\em arXiv: 1812.08415}.

\bibitem{LX19}
L.~Li and X.~Li.
\newblock {Dirichlet forms and polymer models based on stable processes}.
\newblock {\em Stochastic Process. Appl., to appear.}

\bibitem{LY172}
L.~Li and J.~Ying.
\newblock On symmetric linear diffusions.
\newblock {\em Trans. Amer. Math. Soc.}, 371(8):5841--5874, 2019.

\bibitem{Lou}
S.~Lou.
\newblock {Brownian motion with drift on spaces with varying dimension}.
\newblock {\em Stochastic Process. Appl.}, 129(6):2086--2129, 2019.

\bibitem{RY}
D.~Revuz and M.~Yor.
\newblock {\em {Continuous martingales and Brownian motion}}.
\newblock Springer-Verlag, Berlin, 1999.

\bibitem{RT15}
M.~R{\"o}ckner and G.~Trutnau.
\newblock {About the infinite dimensional skew and obliquely reflected
  Ornstein-Uhlenbeck process}.
\newblock {\em Infin. Dimens. Anal. Quantum Probab. Relat. Top.},
  18(4):1550031--1550055, 2015.

\bibitem{S88}
M.~Sharpe.
\newblock {\em {General theory of Markov processes}}.
\newblock Academic Press, Inc., Boston, MA, 1988.

\bibitem{T14}
M.~Takeda.
\newblock Criticality and subcriticality of generalized schrödinger forms.
\newblock {\em Illinois J. Math.}, 58(1):251--277, 2014.

\bibitem{YZ10}
J.~Ying and M.~Zhao.
\newblock {The uniqueness of symmetrizing measure of Markov processes}.
\newblock {\em Proc. Amer. Math. Soc.}, 138(6):2181--2185, 2010.

\bibitem{Z1}
Q.~S. Zhang.
\newblock {Gaussian bounds for the fundamental solutions of $\nabla(A\nabla
  u)+B\nabla u-u_{t}=0$}.
\newblock {\em Manuscripta Math.}, 93(3):381--390, 1997.

\end{thebibliography}


\end{document}